\documentclass[psamsfonts]{amsart}

\usepackage[english]{babel}
\usepackage[utf8x]{inputenc}
\usepackage[T1]{fontenc}
\usepackage[parfill]{parskip}

\usepackage[top=3cm,bottom=3cm,left=3cm,right=3cm,marginparwidth=1.75cm]{geometry}

\usepackage{graphicx}
\graphicspath{C:\\Users\\David\\Documents\\Grad School\\Algebra\\HW10\\}

\usepackage[all,arc]{xy}
\usepackage{enumerate}
\usepackage{mathrsfs}

\newtheorem{thm}{Theorem}[section]
\newtheorem*{thm*}{Theorem}
\newtheorem{cor}[thm]{Corollary}
\newtheorem{proposition}[thm]{Proposition}
\newtheorem{lem}[thm]{Lemma}

\newtheorem{introthm}{Theorem}

\theoremstyle{definition}
\newtheorem{defn}[thm]{Definition}

\newtheorem{notn}[thm]{Notation}

\theoremstyle{remark}
\newtheorem{remark}[thm]{Remark}

\newtheorem{example}[thm]{Example}

\makeatletter
\let\c@refuation\c@thm
\makeatother
\numberwithin{equation}{section}

\usepackage{amsmath,amssymb}
\usepackage{mathtools}
\usepackage[colorlinks=true, allcolors=blue]{hyperref}
\usepackage{amsfonts}
\usepackage[mathscr]{euscript}
\usepackage{array}
\usepackage{caption}
\usepackage{tikz}
\usepackage{tikz-cd}
\usetikzlibrary{arrows, matrix}
\usepackage{cite}
\usepackage{relsize}

\newcommand{\squarebin}{\mathbin{\square}}

\DeclareMathOperator{\id}{id}

\DeclareMathOperator{\Set}{Set}

\DeclareMathOperator{\map}{Map}
\DeclareMathOperator{\core}{Core}
\DeclareMathOperator{\Mack}{Mack}
\DeclareMathOperator{\Tamb}{Tamb}

\newcommand{\liminj}{\varinjlim}

\newcommand{\burnside}{\mathcal{A}}
\newcommand{\Oburnside}{\mathcal{A}_{\mathcal{O}}}
\newcommand{\Tburnside}{\mathcal{P}}
\newcommand{\bitburnside}{\mathcal{P}^G_{\mathcal{O}_m,\mathcal{O}_a}}
\newcommand{\compPair}{(\mathcal{O}_m,\mathcal{O}_a)}
\newcommand{\tCompPair}{(\mathcal{T}_m,\mathcal{T}_a)}
\newcommand{\spec}{Sp}
\newcommand{\ie}{i.e. }

\usepackage{geometry}

\title{Bi-incomplete Tambara Functors as $\mathcal{O}$-commutative monoids}
\author{David Chan}

\begin{document}
	\maketitle
	
\begin{abstract}
	Tambara functors are an equivariant generalization of rings that appear as the homotopy groups of genuine equivariant commutative ring spectra.  In recent work, Blumberg and Hill have studied the corresponding algebraic structures, called bi-incomplete Tambara functors, that arise from ring spectra indexed on incomplete $G$-universes.  In this paper, we answer a conjecture of Blumberg and Hill by proving a generalization of the Hoyer--Mazur theorem in the bi-incomplete setting.  
	
	Bi-incomplete Tambara functors are characterized by indexing categories which parameterize incomplete systems of norms and transfers.  In the course of our work, we develop several new tools for studying these indexing categories.  In particular, we provide an easily checked, combinatorial characterization of when two indexing categories are compatible in the sense of Blumberg and Hill.
\end{abstract}

	\section{Introduction}\label{section:Introduction}
	
	Hill, Hopkins, and Ravenel's solution to the Kervaire invariant one problem \cite{HHR} has motivated a tremendous amount of recent work in developing a deeper understanding of the algebra underlying equivariant stable homotopy theory, such as \cite{Bohmann:Comp,GreenWitt,blumberghillbiincomplete,BlumbergHillIncomplete,HillAndreQuillen,HiHoLocalization,HiHo-ESMS,Hoyer, KMazurThesis}.  In contrast to their non-equivariant counterparts, genuine $G$-spectra can admit many different kinds of $E_{\infty}$-ring structures, parameterized by the $N_{\infty}$-operads of Blumberg and Hill \cite{BlumbergHillIncomplete}.  These same $N_{\infty}$-operads also parameterize a variety of possible additive structures on $G$-spaces and $G$-spectra indexed on incomplete universes. The aim of this paper is to better understand the ways these various flavors of additive and multiplicative structures coexist and interact with one another.
	 
	 For a finite group $G$, recall that a $G$-Mackey functor is a system of abelian groups $M(H)$ indexed on the subgroups $H$ of $G$.  For a pair of subgroups $K\leq H$, the groups $M(K)$ and $M(H)$ are connected by restriction and transfer homomorphisms.  Originally defined to study representation theory, Mackey functors also have a fruitful history as the landing spot of the equivariant cohomology theories first studied by Bredon \cite{Bredon}.   Work of May, and many others \cite{May:Alaska, LMayS}, showed that Mackey functors provide the correct algebraic setting to formulate fundamental concepts like equivariant Poincar\'e duality and $RO(G)$-graded cohomology groups.
	 
	 While Mackey functors should be thought of as the correct analogue of an abelian group in the equivariant setting, the role of commutative rings is played by objects known as Tambara functors.  A Tambara functor $S$ is a Mackey functor such that the abelian group $S(H)$ is a commutative ring for all $H\leq G$, together with multiplicative norm maps $N_K^H\colon S(K)\to S(H)$ for subgroups $K\leq H$.  Norm structures were first studied in connection to equivariant stable homotopy theory by Greenlees and May, and then more systematically by Hill, Hopkins, and Ravenel \cite{GreeMay:Local,HHR}.  If $E$ is a homotopy commutative ring in $G$-spectra, then $\pi_0(E)$ is naturally a Tambara functor.  In fact, work of Ullman has shown that the category of Tambara functors is equivalent to the category of equivariant Eilenberg-MacLane ring spectra \cite{Ullman:Tambara}. 
	 
	The study of Tambara functors has led to interesting developments in the general study of categories with $G$-action.  The category of Mackey functors has a symmetric monoidal product, called the box product.  Equipped with this structure, one might guess that the Tambara functors, in analogy with commutative rings in abelian groups, are the commutative monoids; this is not the case, as the commutative monoids are the Green functors.  Instead, to recover Tambara functors, one has to consider an equivariant symmetric monoidal structure \cite{HiHo-ESMS} on the category of Mackey functors.  The Hoyer--Mazur theorem states that category of Mackey functors admits such an equivariant symmetric monoidal structure, and that the Tambara functors are exactly the equivariant commutative monoids\cite{Hoyer,KMazurThesis}.  The main object of this paper is to prove a conjecture of Blumberg and Hill, extending the Hoyer--Mazur theorem to a statement about \emph{bi-incomplete Tambara functors}.
	
	We begin by describing two different variations of Mackey and Tambara functors.  The first variant is the \emph{incomplete Tambara functors} of Blumberg and Hill \cite{BlumbergHillIncomplete}, which are Tambara functors $S$ that only have some of their norms maps.  These structures arise naturally when studying localizations of equivariant ring spectra.  In particular, if $E$ is a commutative ring in $G$-spectra, there can exist localizations $\tilde{E}$ of $E$ that are \emph{not} commutative rings.  In particular, $\pi_0(\tilde{E})$ is not a Tambara functor because the process of localization makes some of the norms impossible to define.  The observation of Blumberg and Hill is that while some of the norms are lost, some of the norms remain, implying that $\pi_0(\tilde{E})$ is an incomplete Tambara functor.
	
	The second variant we need is the \emph{incomplete Mackey functors}.  In analogy with the incomplete Tambara functors, an incomplete Mackey functor is simply a Mackey functor which only has some of its transfer maps.  Such structures were first studied, to our knowledge, by Lewis as a natural home for the unstable equivariant homotopy groups of $G$-spaces, as well as homotopy groups of $G$-spectra indexed on incomplete $G$-universes \cite{Lewis:Hurewicz}.  Recently, these structures have been studied by Blumberg and Hill in connection to novel incomplete models for equivariant stable homotopy categories \cite{BHIncompleteStable}.
	
	Both ``incomplete'' settings arise from the restriction of some collections of operations, norms or transfer, allowed in our algebraic structures.  A bi-incomplete Tambara functor is a Tambara functor in which we are only allowed some of the norms and some of the transfers.  These objects were introduced by Blumberg and Hill in \cite{blumberghillbiincomplete} with examples coming from ring spectra indexed on incomplete $G$-universes.  In the same paper, Blumberg and Hill conjectured an incomplete version of the Hoyer--Mazur theorem.  Our first main theorem, which we state imprecisely here and more precisely in the main text, answers this conjecture in the affirmative.

	\begin{introthm}[Appears below as Theorem \ref{thm:secondMainTheorem}]\label{thm:introThm1}
		Any category of (additively) incomplete Mackey functors admits an equivariant symmetric monoidal structure for every compatible collection of multiplicative norm maps.  In each of these structures, the equivariant commutative monoids are exactly the bi-incomplete Tambara functors.
	\end{introthm}
	
	The proof of Theorem \ref{thm:introThm1} requires understanding the interplay of various incomplete systems of norms and transfers. Because the norms and transfers of a bi-incomplete Tambara functor must interact in non-trivial ways, we cannot pick arbitrary collections of available norms and transfers and expect to get a coherent and useful algebraic object.  We keep track of available norms and transfers of a bi-incomplete Tambara functor by using \emph{indexing categories}, written $\mathcal{O}_{norm}$ and $\mathcal{O}_{transfer}$, which are certain subcategories of the category of finite $G$-sets.  Indexing categories were first studied in the context of incomplete Tambara functors by Blumberg and Hill\cite{BlumbergHillIncomplete}.   A pair of indexing categories $(\mathcal{O}_{norm},\mathcal{O}_{transfer})$  are called \emph{compatible} if they are able to index a bi-incomplete Tambara functor.  
	
	The first step toward proving Theorem \ref{thm:introThm1} is to show that every compatible pair of indexing categories determines an equivariant symmetric monoidal structure on the category of incomplete Mackey functors with transfers indexed by $\mathcal{O}_{transfer}$.  The construction requires us to define \emph{algebraic norm functors} which allow us to move between $K$-Mackey functors and $H$-Mackey functors for subgroups $K\leq H$.  These norm functors are an algebraic analogue of the famous Hill--Hopkins--Ravenel norm on equivariant spectra.  In the setting of ordinary Mackey functors, the algebraic norms were first constructed by Mazur in the case $G=C_{p^n}$, and later for all finite groups by Hoyer.  These norms play an important role in other constructions in equivariant algebra such as generalizations of topological Hochschild homology for Green functors \cite{GreenWitt}.
	
	Constructing the algebraic norm functors requires understanding exactly the conditions under which two pairs of indexing categories can be compatible.  Blumberg and Hill have provided one characterization which depends heavily on computation of the coinduction functors $\map_K(H,-)\colon\Set^K\to \Set^H$ for subgroups $K\leq H\leq G$. Unfortunately it is, in practice, not a simple task to determine the orbit decomposition coinduced sets.  Much of the difficulty arises from the fact that the coinduction functor does not preserve disjoint unions.
	
	The second main theorem of this paper provides a way to check that two indexing categories are compatible without needing to compute any coinduction functors.  We achieve this by reframing the problem in terms of \emph{transfer systems}, a tool developed independently by Rubin and Balchin--Barnes--Roitzheim to give an alternate description of the combinatorics of indexing categories \cite{BalchinBarnesRoitzheim, rubin_detecting}.  In Definition \ref{defn:compatibleTransferSystems} we give a notion of a compatible pair of transfer systems which can be checked without any difficult computations.  Our second main theorem is:
	
	\begin{introthm}[Appears below as Theorem \ref{thm:catEqualSystem}]\label{thm:introThm2}
		Compatibility of a pair of indexing categories is equivalent to a compatibility condition (Definition \ref{defn:compatibleTransferSystems}) on the associated pair of transfer systems. 
	\end{introthm}

	Theorem \ref{thm:introThm2} provides enough understanding of compatible pairs to define the norm functors, but it is also of independent interest.  Important examples of compatible pairs of indexing categories arise naturally from equivariant Steiner and linear isometries operads indexed on $G$-universes.  Rubin has, in some cases, classified the kinds of indexing categories that can arise from the equivariant Steiner operads, but a classification of those determined by linear isometries operads remains unknown \cite{rubin_detecting}.  Theorem \ref{thm:introThm2} could provide an avenue for understanding the linear isometries operads via their compatibility with the Steiner operads.
	
	The paper is organized as follows.  In Section \ref{section:indexingCategories} we review the relevant background on indexing systems and indexing categories.  Most of this section is review, although there are some new results that show how one can make use of the normal core of a subgroup $H\leq G$ when analyzing indexing categories.  These new results play a key role in the characterization of compatible transfer systems.
	
	In Section \ref{section:biincomplete}, we recall the definition of bi-incomplete Tambara functors from \cite{blumberghillbiincomplete} and provide examples.  In Section \ref{section:transferSystems} we discuss transfer systems and define the notion of compatible transfer systems, proving it agrees with that of indexing categories.  In Section \ref{section:GCommMonoids}, we review the definitions of $G$-symmetric monoidal Mackey functors and $\mathcal{O}$-commutative monoids which are our model for equivariant symmetric monoidal categories.  
	
	In Section \ref{section:biincompleteAsMonoids} we use the results of Section \ref{section:transferSystems} to construct the norm functors on incomplete Mackey functors. These norms are then used to prove that incomplete Mackey functors assemble into a $G$-symmetric monoidal Mackey functor.  
	
	In Section \ref{section:TambaraNorms} we define a norm functor for Tambara functors and provide a comparison between this and the norms of Mackey functors. This comparison is essential for translating between bi-incomplete Tambara functors and the equivariant commutative monoids of incomplete Mackey functors.  We use this in Section \ref{section:MainTheorem} to prove our generalized Hoyer--Mazur Theorem.  The paper concludes in Section \ref{section:technicalProofs} which contains the proofs of some technical lemmas used in the paper.
	
	\textbf{Acknowledgments:} The author would like to thank Anna Marie Bohmann for many helpful conversations and suggestions on how to improve the paper.  He would also like to thank Andy Jarnevic and Aidan Lorenz for conversations that helped to improve the exposition.  This research was partially supported by NSF-DMS grant 2104300.
	
	\section{Indexing Categories}\label{section:indexingCategories}
	
	Throughout this paper, systems of transfers and norms for Mackey and Tambara functors for a finite group $G$ are indexed by \emph{indexing categories}.  We begin this section by recalling the definitions and some basic examples.  At the end of the section we prove some new results regarding the interaction of indexing categories with the normal cores of subgroups $H$ of a group $G$.  
	
	\begin{defn}
		An \emph{indexing category} $\mathcal{O}$ is a wide, pullback stable, finite coproduct complete subcategory of $\Set^G$.  For any $H\leq G$, an $H$-set $X$ is called \emph{$\mathcal{O}$-admissible} if there is a morphism $G\times_H X \to G/H$ in $\mathcal{O}$.
	\end{defn}

	\begin{example}
		For any group $G$ we have the complete indexing category $\mathcal{O}^{gen}=\Set^G$.  
	\end{example}

	\begin{example}\label{ex:trivialIndexing}
		For any group $G$ we have the trivial indexing category $\mathcal{O}^{tr}$.  A map $f\colon X\to Y$ is in $\mathcal{O}^{tr}$ if and only if for any orbit $X_0\subset X$, the restriction $f\colon X_0\to Y$ is an isomorphism onto its image.  Essentially, the only maps in $\mathcal{O}^{tr}$ are fold maps and inclusions.
	\end{example}

	\begin{example}\label{ex:restrictedIndexingCat}
		Let $H\leq G$ be a subgroup.  Every $G$-indexing category $\mathcal{O}$ determines an $H$-indexing category $i_H^*\mathcal{O}$.  A map of $H$-sets $f\colon S\to T$ is in $i_H^*\mathcal{O}$ if the induced $G$-map $G\times_{H}f\colon G\times_HS\to G\times_HT$ is in $\mathcal{O}$.  It follows from the definitions that for any $K\leq H$, the $i_H^*\mathcal{O}$-admissible $K$-sets are the same at the $\mathcal{O}$-admissible $K$-sets.  Note that for any $K\leq H$, the induction functor $H\times_K(-)\colon \Set^K\to \Set^H$ restricts to a functor $H\times_K(-)\colon i_K^*\mathcal{O}\to i^*_H\mathcal{O}$.
	\end{example}

	Our definition of an $\mathcal{O}$-admissible set is not exactly the same as one given in Blumberg and Hill \cite{blumberghillbiincomplete}, though it is equivalent.  The difference is that Blumberg and Hill define admissibility for \emph{indexing systems}, which are structures that carry data equivalent to that of indexing categories.  Indexing systems play no explicit role in this paper, and so we have worded our statements in terms of indexing categories. In keeping with this, the next three lemmas are not new, although we include proofs because we do not know a reference in which they are stated in this language.

	\begin{lem}[Proposition 3.1 of \cite{BlumbergHillIncomplete}]\label{closedUnderSubobjects}
		Suppose $\mathcal{O}$ is an indexing category and $H\leq G$ is a subgroup. If $X$ is an $\mathcal{O}$-admissible $H$-set then every orbit of $X_0\subset X$ is also $\mathcal{O}$-admissible.
	\end{lem}
	\begin{proof}
		
		By assumption, we have a map $f\colon G\times_H X\to G/H$ in $\mathcal{O}$ which allows us to construct the composite
		\[
			G\times_HX_0\to G\times_H X\xrightarrow{f} G/H
		\]
		where the first map is the inclusion.  By Proposition 3.1 of \cite{BlumbergHillIncomplete}, all monomorphisms of $\Set^G$ are in $\mathcal{O}$ and thus we have constructed a map $G\times_HX_0\to G/H$ that is in $\mathcal{O}$.
	\end{proof}
	
	\begin{lem}[Proposition 3.3 of \cite{rubin_detecting}] \label{intersectionLemma}
		Let $\mathcal{O}$ be any indexing category. If $H/K$ is $\mathcal{O}$-admissible, then for all subgroups $L\leq G$ we have that $(H\cap L)/(K\cap L)$ is also $\mathcal{O}$-admissible.
	\end{lem}
	\begin{proof}
		Since $H\cap L = H\cap (H\cap L)$ we can replace $L$ by $H\cap L$ and reduce to the case where $L \leq H$.  Unwinding the definitions, it suffices to construct a map $G/(L\cap K)\to G/L$ in $\mathcal{O}$.  By assumption, we have a map $f\colon G/K\to G/H$ in $\mathcal{O}$.  Consider the following pullback diagram:
		\[
			\begin{tikzcd}
				G/L\times_{G/H}G/K \ar{r}\ar["p"]{d} & G/K\ar["f"]{d}\\
				G/L \ar{r} & G/H
			\end{tikzcd}
		\]
		The map $p$ must be in $\mathcal{O}$ by pullback stability and the fact that $f$ is in $\mathcal{O}$.  By Lemma \ref{closedUnderSubobjects}, we are done if we can show the pullback $G/L\times_{G/H}G/K$ has an orbit isomorphic to $G/(L\cap K)$.  This is immediate since the element $(eK,eL)$ in the pullback has stabilizer exactly $L\cap K$. 
	\end{proof}

	When applying Lemma \ref{intersectionLemma} we say $(H\cap L)/(K\cap L)$ is obtained from $H/K$ by intersection with $L$.  In this language, the lemma can be summarized by saying the collection of $\mathcal{O}$-admissible sets is closed under intersection with subgroups.
	
	Our next lemma states that the induction functors associated to subgroups $K\leq H$, with $H/K$ an $\mathcal{O}$-admissible $H$-set, give a function from admissible $K$-sets to admissible $H$-sets.  Blumberg and Hill refer to this property as closure under self-induction.
	
	\begin{lem}[Closure under self-induction]\label{lem:closureUnderSelfInduction}
		Suppose $K\leq H$ are subgroups and $H/K$ is an $\mathcal{O}$-admissible $H$-set.  If $T$ is an $\mathcal{O}$-admissible $K$-set, then $H\times_K T$ is an $\mathcal{O}$-admissible $H$-set.
	\end{lem}
	\begin{proof}
		The assumptions guarantee maps $f\colon G\times_K T\to  G/K$ and $g\colon G\times_H(H/K)\to G/H$ in $\mathcal{O}$.  Since there is an isomorphism $i\colon G/K\cong G\times_H(H/K)$, which must be in $\mathcal{O}$, we can for the composite $g\circ i \circ f\colon G\times_K T\to G/H$, which is a morphism in $\mathcal{O}$.  The result now follows from the isomorphism $G\times_K T\cong G\times_H(H\times_K T)$. 
	\end{proof}

	If $S$ is a $G$-set, we write $G_s$ for the stabilizer of $s$ in $G$.

	\begin{lem}[See Section 3 of\cite{BlumbergHillIncomplete}]\label{lem:stabilizerQuotients}
		A map $f\colon S\to T$ is in $\mathcal{O}$ if and only if for any $s\in S$ we have that $G_{f(s)}/G_S$ is an $\mathcal{O}$-admissible $G_{f(s)}$-set.
	\end{lem}

	In the remainder of this section we make a few new observations on the structure of indexing categories.  Of particular interest is the interaction between indexing categories and the \emph{cores} (Definition \ref{defn:core} below) of various subgroups $H\leq G$.  We begin with a modest generalization of Lemma \ref{intersectionLemma}.
	
	\begin{cor}\label{intersectionLemma2}
		If $H_1/K_1$ and $H_2/K_2$ are two $\mathcal{O}$-admissible sets, then $(H_1\cap H_2)/(K_1\cap K_2)$ is also $\mathcal{O}$-admissible.
	\end{cor}
	\begin{proof}
		Intersecting $H_1/ K_1$ with $H_2$ yields that $(H_1\cap H_2)/(K_1\cap H_2)$ is $\mathcal{O}$-admissible.  Similarly, intersecting $H_2/K_2$ with $K_1$ gives us that $(K_1\cap H_2)/(K_1\cap K_2)$ is $\mathcal{O}$-admissible.  We are now done, as 
		\[
			(H_1\cap H_2)/(K_1\cap K_2) \cong (H_1\cap H_2)\times_{K_1\cap H_2}(K_1\cap H_2)/(K_1\cap K_2)
		\]
		with the right hand side $\mathcal{O}$-admissible by Lemma \ref{lem:closureUnderSelfInduction}.
	\end{proof}

	\begin{defn}\label{defn:core}
		For a subgroup $H\leq G$ the \emph{core} of $H$ in $G$, denoted $\core_G(H)$, is the intersection of all conjugates of $H$.  That is,  
		\[
			\core_G(H) = \bigcap\limits_{g\in G} H^g
		\]	
		where $H^g = gHg^{-1}$.
	\end{defn}

	\begin{remark}
		The core of a subgroup admits several equivalent definitions.  In particular, it can be described as the kernel of the group homomorphism $G\to \Sigma_{|G/H|}$ which realizes the left $G$-action on the set of $H$-cosets.   While the definition above is convenient for indexing categories, especially in light of Corollary \ref{intersectionLemma2}, this second description can be helpful to have in mind later when considering the coinduction functors $\map_H(G,-)\colon \Set^H\to \Set^G$.
	\end{remark}

	\begin{lem}\label{closureUnderConjugation}
		Suppose $H/K$ is an $\mathcal{O}$-admissible $H$-set.  Then for all $g\in G$ the sets $H^g/K^g$ are $\mathcal{O}$-admissible.
	\end{lem}
	\begin{proof}
		The result follows from noticing that
		\[
			G\times_H H/K \cong G/K \cong G/K^g \cong G\times_{H^g}H^g/K^g
		\]
		and the fact that indexing categories contain all isomorphisms.
	\end{proof}

	\begin{proposition}\label{coreLemma}
		Suppose $\mathcal{O}$ is an $G$-indexing category and $K\leq H\leq G$ is a chain of subgroups.  If $H/K$ is $\mathcal{O}$-admissible then so is $\core_G(H)/\core_G(K)$.  In particular, if $G/H$ is admissible then so is $G/\core_G(H)$.
	\end{proposition}

	\begin{proof}
		The result follows from repeated applications of Lemma \ref{closureUnderConjugation} and Corollary \ref{intersectionLemma2}.
	\end{proof}

	Proposition \ref{coreLemma} is useful when working with indexing categories  because it allows us to effectively replace an arbitrary subgroup $H\leq G$ with the $\core_G(H)$, which is a normal subgroup.  This is especially helpful when dealing with coinduced $G$-sets.  For now, we content ourselves with the following example.

	\begin{example}
		Suppose $G=\Sigma_n$ is the symmetric group on $n$ elements for $n\geq 5$.  Let $H$ be any proper subgroup of $\Sigma_n$, other than the alternating group.  Since $\core_{\Sigma_n}(H)$ is a normal subgroup of $\Sigma_n$ that is contained in $H$ it must be the trivial group. By Proposition \ref{coreLemma}, if $\mathcal{O}$ is an indexing category with $\Sigma_n/H$ admissible then $\Sigma_n/e$ is also admissible.   	
	\end{example}

	The previous example illustrates the surprising ways the core can affect the admissible sets of an indexing category.  In Section \ref{section:transferSystems}, specifically the proof of Proposition \ref{prop:catImpliesSystem}, the interaction between the cores of subgroups and indexing categories plays a central technical role in our characterization of \emph{compatible} indexing categories.  We will define compatibility of indexing categories precisely in Definition \ref{defn:compatibleIndexingCategories} below.  Presently, it is sufficient to think about a pair of indexing categories $\compPair$ as being compatible if $\mathcal{O}_m$ acts on a $\mathcal{O}_a$ in some algebraic sense.  We end this section by collecting a few corollaries illustrating the ways that the core sheds light on the rigid structure of compatible pairs of indexing categories.
	
	\begin{cor}\label{cor:simpleGroups}
		Suppose $H$ is a subgroup of $G$ such that $\core_G(H)=e$ is the trivial group and suppose $\mathcal{O}_m$ is an indexing category with $G/H$ an $\mathcal{O}_m$-admissible $G$-set.  If $\compPair$ is a compatible pair of indexing categories, then $\mathcal{O}_a$ is the complete indexing category.
 	\end{cor}
 
	\begin{proof}
		By Proposition \ref{coreLemma}, we have that $G/e$ is $\mathcal{O}$-admissible.  The result is now the conclusion of Corollary 6.5 of\cite{blumberghillbiincomplete}.
	\end{proof}
	\begin{cor}
		Suppose $G$ is a simple group, $H\leq G$ a proper subgroup, and $\mathcal{O}_m$ is an indexing category.  If $G/H$ is $\mathcal{O}_m$-admissible then the only indexing category $\mathcal{O}_a$ which is compatible with $\mathcal{O}_m$ is the complete indexing category.
	\end{cor}

	\section{Bi-incomplete Tambara Functors}\label{section:biincomplete}
	
	In equivariant algebra, the role of abelian groups is played by algebraic objects known as Mackey functors. These amount to a collection of abelian groups, indexed on the subgroups of $G$, that are connected by various maps called transfers and restrictions.  The role of commutative rings is played by Tambara functors which, in addition to being Mackey functors, have the structure of a ring for every subgroup as well as multiplicative norm maps.  In certain cases, only some of the possible transfer and norm maps are available, and we call such structures bi-incomplete Tambara functors.  In this section we review the notion of incomplete Mackey functors and bi-incomplete Tambara functors.  Most of the material from this section can be found in \cite{blumberghillbiincomplete}.  Throughout we let $\mathcal{O}$ be an indexing category.
	
	\begin{defn}
		For a finite group $G$, the \emph{$\mathcal{O}$-Burnside Category} $\mathcal{A}^G_{\mathcal{O}}$ is the category whose objects are finite $G$-sets and morphisms $\mathcal{A}^G_{\mathcal{O}}(X,Y)$ are isomorphism classes of spans 
		\[
			[X \xleftarrow{r} A \xrightarrow{t} Y]
		\]  
		where $r,t$ are equivariant maps of $G$-sets with $t\in \mathcal{O}$. Composition is given by pullback, i.e.\ the composition of $[X \xleftarrow{r_1} A \xrightarrow{t_1} Y]$ and $[Y \xleftarrow{r_2} B \xrightarrow{t_2} Z]$ is the class of the span along the top of
		\[
			\begin{tikzcd}
				& & A\times_Y B \ar[bend right, "r_1\circ \pi_1"']{ddll} \ar["\pi_1"]{dl}\ar[bend left,"t_2\circ \pi_2"]{ddrr} \ar["\pi_2"']{dr} & & \\
				& A \ar["t_1"']{dr} \ar["r_1"]{dl} & & B \ar["t_2"']{dr} \ar["r_2"]{dl} & \\
				X & & Y & & Z	
			\end{tikzcd}
		\]
		where the middle square is a pullback.
	\end{defn} 
	
	\begin{remark}
		Composition in the category $\mathcal{A}^G_{\mathcal{O}}$ is well defined because the indexing category $\mathcal{O}$ is pullback stable.  In many ways this is the point of indexing categories, although they were invented in \cite{BlumbergHillIncomplete} to manage incomplete systems of norms, not transfers.
	\end{remark}
	
	\begin{example}
		When $\mathcal{O} = \Set^G$ is the complete indexing category, we refer to $\burnside^G_{\mathcal{O}}$ as simply the Burnside category of $G$ and denote it by $\burnside^G$.
	\end{example}

	Morphisms in the $\mathcal{A}^G_{\mathcal{O}}$ factor as the composition of two nice families of maps called the restrictions and transfers.  For a map $f\colon X\to Y$ in $\Set^G$, we define the restriction of $f$ by 
	\[
		R_f = [Y\xleftarrow{f} X =X].
	\]
	Similarly, if $f\in \mathcal{O}$ we can define the transfer of $f$ by
	\[
		T_f = [X = X \xrightarrow{f} Y].
	\]
	
	Using these, we can decompose an arbitrary span as 
	\begin{equation}\label{spanDecomposition}
		[X \xleftarrow{f} A \xrightarrow{g} Y] = T_g\circ R_f.
	\end{equation}
	
	When $f\colon G/K\to G/H$ for $K\leq H$ is the canonical quotient we denote $R_f$ and $T_f$ by $R^H_K$ and $T^H_K$ respectively.  Finally, if $f\colon G/H\to G/H^g$ is a conjugation isomorphism, we denote $T_f$ by $c_g$, the conjugation by $g$. Choosing $c_{g^{-1}} = R_f$ leads to the same conjugation maps.
	
	The $\mathcal{O}$-Burnside category of a group $G$ is a semi-additive category with finite products given by disjoint union of $G$-sets.  It follows that a product preserving functor $M\colon \Oburnside^G\to \Set$ naturally takes values in commutative monoids.  For any finite $G$-set $X$, the addition on $M(X)$ is given by 
	\[
		M(X)\times M(X)\cong M\left(X\amalg X\right)\xrightarrow{M(T_{\nabla})} M(X)
	\]
	where $\nabla\colon X\amalg X\to X$ is the fold map.  The unit is given by the image of 
	\[
		M(\emptyset)\xrightarrow{M(T_{\emptyset\to X})} M(X)
	\]
	Note $M(\emptyset)$ must be a singleton set as $M$ is product preserving.
	
	\begin{defn}
		A \emph{semi $\mathcal{O}$-Mackey functor} is a product preserving functor $M\colon \Oburnside^G\to \Set$.  An \emph{$\mathcal{O}$-Mackey functor} is a semi $\mathcal{O}$-Mackey functor that is group complete, in the sense that for all $X$ the commutative monoid $M(X)$ is actually an abelian group. Any semi $\mathcal{O}$-Mackey functor $M$ determines a $\mathcal{O}$-Mackey functor we denote $M^+$ by group completing $M(X)$ at every $X$.  
		
		A morphism of (semi) $\mathcal{O}$-Mackey functors is a natural transformation of product preserving functors.  We denote the category of $\mathcal{O}$-Mackey functors for a group $G$ by $\Mack_{\mathcal{O}}^G$.  	When $\mathcal{O} = \Set^G$ is the complete indexing category we recover the usual category of Mackey functors and denote this category by $\Mack^G$.
	\end{defn}
	
	\begin{remark}
		It is common to define Mackey functors to be product preserving functors from the Burnside category into the category of abelian groups.  While this definition is equivalent, we have elected to value our Mackey functors in the category of sets for two reasons.  First, this definition lines up more readily with the view that the Burnside category is the multisorted Lawvere theory whose models are Mackey functors. Second, and more importantly, we wish to have a better analogy with Tambara functors which are also given as product preserving functors from a specified category $\mathcal{P}^G$ into sets.  Unlike Mackey functors, however, Tambara functors cannot be defined as functors that land in abelian groups.  The essential difference is that Tambara functors are a model for equivariant ring structures and the bilinear multiplication map $R\times R\to R$ of a ring $R$ is not a map of abelian groups. 
	\end{remark}
	
	While there are many objects in $\burnside^G_{\mathcal{O}}$, it suffices to specify the value of a Mackey functor $M$ on only finitely many objects.  Since every finite $G$-set is the union of its orbits, and all orbits are isomorphic to $G/H$ for some subgroup $H$, an $\mathcal{O}$-Mackey functor $M$ is determined by its value on the objects $G/H$ as well as on spans $[G/H\xleftarrow{r} G/K \xrightarrow{t} G/L] = T_K^LR^H_K$.  Abusing notation, we use $R^H_K$ (resp.\ $T^H_K$) to denote both the restriction (resp.\ transfer) in $\burnside^G_{\mathcal{O}}$ as well as the map induced by applying $M$.  
	
	The following lemma gathers together several standard and useful facts about $\mathcal{O}$-Mackey functors.  The proofs are all standard and, with the exception of the double coset formula, straightforward.  The double coset formula follows from a computation of the orbit decomposition for the product of transitive $H$-sets: $H/L\times H/K$.
	
	\begin{lem}\label{lem:propertiesOfMackeyFunctors}
		Let $H$ be a subgroup of $G$. The restriction and transfer maps of a semi-$\mathcal{O}$-Mackey functor $M$ satisfy the following properties, assuming the transfers exists.
		\begin{enumerate}
			\item (Additivity) All transfers and restrictions are additive maps of commutative monoids.
			\item (Composition) For $K\leq L\leq H$ we have $R^L_KR^H_L = R^H_K$  and  $T^H_LT^L_K = T^H_K$.
			\item (Double Coset Formula) Let $K$ and $L$ be subgroups of $H$ and $\gamma_1,\dots, \gamma_n$ be a choice of representatives for the double cosets $K\backslash H/L$.  If $H/L$ is $\mathcal{O}$-admissible, we have the following equality of maps:
			\[
				R^H_KT^H_L = \sum\limits_{\gamma_i} T_{K\cap L^{\gamma_i}}^K R^{L^{\gamma_i}}_{K\cap L^{\gamma_i}} c_{\gamma_i}
			\]
			\item (Weyl Group Action) The Weyl group $W_G(H) = N_G(H)/H$ acts naturally on the $G$-set $G/H$.  Applying $M$ to this action gives $M(G/H)$ a canonical $W_G(H)$ action.  
		\end{enumerate}
	\end{lem}

	\begin{remark}\label{rem:mackeyDoubleCoset}
		In the statement of the double coset formula we note that $\mathcal{O}$-admissibility of $H/L$ is sufficient to imply both sides of the equation are well-defined.  The left hand side is clear. To see the right hand side is well defined, note that $K/(K\cap L^{\gamma_i})$ is admissible for all $i$ by Lemma \ref{intersectionLemma}.  As a warning, it appears to be possible for all the transfer on the right hand side to be well defined but for the transfer on the left to still be undefined.  
	\end{remark}

	Examples of Mackey functors include the Burnside Mackey functor, representation rings, and Galois field extensions discussed below.  For simplicity, we write the examples as complete Mackey functors, i.e. with $\mathcal{O} = \Set^G$.  Incomplete versions can be obtained for any $\mathcal{O}$ by restricting our functors along the inclusion $\Oburnside^G\to \burnside^G$.
	\begin{example}\label{ex:represented}
		For any object $X\in \burnside^G$ there is a represented semi-Mackey functor $A_X = \burnside^G(X,-)$.   The group completion $A^+_{G/G}$ is called the Burnside Mackey functor.  One can check that $A_{G/G}(G/H)$ is isomorphic to the set of isomorphism classes of finite $H$-sets with addition given by disjoint union.  It follows that $A^+_{G/G}(G/H)$ is the free abelian group on the set of transitive $H$-sets.  The restriction $R^H_K$ is given by the restriction of $H$-sets to $K$-sets while the transfer $T^H_K$ is given by induction.
	\end{example}

	\begin{example}\label{ex:repRing}
		For every $H\leq G$, we can form the abelian group of virtual complex $H$-representations $\textrm{Rep}(H)$.  For subgroups $K\leq H$, these groups are related via the classical induction and restriction homomorphisms $\textrm{Ind}_K^H\colon \textrm{Rep}(K)\to \textrm{Rep}(H)$ and $\textrm{Res}^H_K\colon \textrm{Rep}(H)\to \textrm{Rep}(K)$.  This produces the complex representation ring Mackey functor $R_G$, which is given by $R_G(G/H) = \textrm{Rep}(H)$ with transfers and restrictions are given by induction and restriction homomorphisms. Varying the kinds of representations considered leads to other examples, such as the real orthogonal representation ring $RO_G$, the unitary representation ring $RU_G$, etc.
	\end{example}
	\begin{example}\label{galoisExample}
		Let $F\subset L$ be a a finite Galois field extension with Galois group $G$.  We obtain a Mackey functor $\mathbb{L}$ given by $\mathbb{L}(G/H) = L^H$.  The restriction $R^H_K$ associated to $K\leq H$ is given by inclusion of fixed points.  The transfers $T^H_K$ are given by the classical field trace:  picking representatives $h_1K,\dots, h_nK$ for the left cosets $H/K$ we define $\textrm{tr}\colon L^K\to L^H$ by 
		\[
			\textrm{tr}(x) = \sum\limits_{i=1}^n  h_ix
		\]
		which does not depend on the choice of representatives as $x\in L^K$.  In fact, the same transfers and restrictions give a Mackey functor for any ring $R$ with $G$-action.
	\end{example}

	Examples of incomplete Mackey functors come to us from unstable equivariant homotopy theory.
	\begin{example}
		Recall the trivial indexing category $\mathcal{O}^{tr}$ of Example \ref{ex:trivialIndexing}.  An $\mathcal{O}^{tr}$-Mackey functor is more commonly known as a coefficient system.  Such an object consists of a system of abelian groups, indexed by the subgroups of $G$, and connected by only the restriction maps.  Given a based $G$-space $X$, we have the $n$-th homotopy coefficient system of $X$ given by $G/H\mapsto \pi_n(X^H)$ where $X^H$ denotes the $H$-fixed points of $X$.  For subgroups $K\leq H$, the restriction map $R^H_K\colon \pi_n(X^H)\to \pi_n(X^K)$ is induced by the inclusion of fixed points $X^H\to X^K$.
	\end{example}
		
	\begin{example}
		Suppose $X$ is a (based) $G$-space which is also an infinite loop space.  In the non-equivariant setting, such a space is equivalent to a connective spectrum; the failure of $X$ to be a genuine $G$-spectrum is measured by the non-existence of transfer maps in the homotopy coefficient systems.  More precisely, for a genuine $\Omega$-$G$-spectrum $Y$, the homotopy groups $\pi_n(Y)$ fit into a Mackey functor with the transfers determined by the delooping isomorphisms $\Omega^V(Y_{n+V})\cong Y_n$.  To build in transfer maps to the homotopy coefficient system of $X$ is equivalent to making choices for equivariant deloopings $X\cong \Omega^V(X_V)$ for various non-trivial $G$-representations $V$.  Such deloopings can be difficult to construct, although models in the case $G=C_2$ have been constructed by Liu, and used in computations of equivariant cohomology by Petersen \cite{Liu_twistedBar,Petersen}.  For more on the connections between incomplete Mackey functors and stability in topology we refer the reader to the work of Lewis and Blumberg--Hill \cite{Lewis:Hurewicz,BHIncompleteStable}.
	\end{example}

	Examples \ref{ex:represented}, \ref{ex:repRing}, and \ref{galoisExample} above are all classical examples of Mackey functors.  Notice, though, that all three examples actually have more structure than simply an abelian group for all subgroups $H\leq G$.  If $M$ is any Mackey functor from the examples,  $M(G/H)$ is actually a ring for all $H$.  This ring structure, and the norm maps to be described, leads to the notion of Tambara functors, originally defined in \cite{Tambara}.  Tambara functors, similar to Mackey functors, are defined by first giving an additive category $\mathcal{P}^G$ and letting  (semi) Tambara functors be the product preserving functors from this category into sets.  
	
	Eventually we need to define bi-incomplete Tambara functors, which depend on a pair of indexing categories $\compPair$, one for the multiplicative norms and one for the additive transfers.  Since the composition laws in the category $\mathcal{P}^G$ are a bit elaborate, we first give the definition in the case of complete indexing categories. We then proceed to explain the necessary modifications to create a category we denote $\bitburnside$ which is the domain of a bi-incomplete Tambara functor.
	
	For a group $G$, the \emph{polynomial category} $\Tburnside^G$ is the category whose objects are finite $G$-sets and with morphism sets $\Tburnside^G(X,Y)$ consisting of isomorphism classes of \emph{bispans} 
	\[
		[X \xleftarrow{r} A \xrightarrow{n} B \xrightarrow{t} Y]
	\]  
	where $r$, $n$, and $t$ are equivariant maps of $G$-sets.  The composition of morphisms is described below. This category, like the Burnside category, is semi-additive and thus for any product preserving functor $S\colon \Tburnside^G\to \Set$ and any $X\in \Tburnside^G$, the set $S(X)$ is naturally a commutative monoid with addition given by
	\[
		S(X)\times S(X)\cong S(X\amalg X) \xrightarrow{S([X\amalg X\xleftarrow{=}X\amalg X\xrightarrow{=}X\amalg X\xrightarrow{\nabla}X])} S(S).
	\]
	
	Putting the fold map $\nabla$ in the middle of the bispan, instead of on the right, leads to a second map $S(X)\times S(X)\to S(X)$ which we call the multiplication.  The composition laws of $\mathcal{P}^G$ are such that the addition and multiplication make $S(X)$ into a commutative semiring (ring possibly without additive inverses).
	
	\begin{defn}
		A \emph{semi-Tambara functor} is a a product preserving functor $\Tburnside^G\to \Set$.  A \emph{Tambara functor} is a semi-Tambara functor $S$ such that $S(X)$ is a ring for all $X$. A morphism of Tambara functors is a natural transformation of product preserving functors and the category of Tambara functors for a group $G$ is denoted by $\Tamb^G$.
	\end{defn} 
	
	Before describing the composition laws for $\Tburnside^G$, we highlight the key difference between the categories $\burnside^G$ and $\Tburnside^G$, namely the extra map $A\xrightarrow{n}B$ in the bispan.  The map $n$ in this category exists to parameterize the norm maps $N^H_K\colon S(G/K)\to S(G/H)$ for Tambara functors $S$.  These norm maps are multiplicative maps between the various rings comprising a Tambara functor.  Between the transfers, norms, and restrictions the category $\Tburnside^G$ has three distinguished kinds of morphisms associated to any map $f\colon X\to Y$ in $\Set^G$.  We denote these by
	\begin{align*}
		R_f  & = [Y\xleftarrow{f} X = X =X], \\
		N_f  & = [X = X \xrightarrow{f} Y = Y], \\
		T_f  & = [X =X = Y\xrightarrow{f} Y],
	\end{align*}
	and refer to these maps as the restriction, norm, and transfer associated to the map $f$.  Just as we did for Mackey functors, we abuse notation and use $R_f, N_f,$ and $T_f$ for both the maps in $\Tburnside^G$ as well as for the maps induced by a Tambara functor $S$.  Further, when $f\colon G/K\to G/H$ is a canonical quotient for $K\leq H$, we write $R^H_K$, $T^H_K$ and $N^H_K$.
	
	To describe the composition laws in $\Tburnside^G$, we first note than an arbitrary bispan 
	\[
		\omega = [X \xleftarrow{r} A \xrightarrow{n} B \xrightarrow{t} Y]
	\]
	can be written as the composition $T_t\circ N_n \circ R_r$.  Given another bispan 
	\[
		\omega' = [Y \xleftarrow{r'} C \xrightarrow{n'} D \xrightarrow{t'} Z],
	\]
	we need to know how to write $\omega'\circ \omega = T_{t'}N_{n'}R_{r'}T_tN_nR_r$ as a composite of the form $T_{t''}N_{n''}R_{r''}$.  To do this, we need to describe the interchange rules relating the various $T$'s, $N$'s and $R$'s; these are summarized over the course of a few lemmas.
	
	\begin{lem}
		For $G$-equivariant maps $f\colon X\to Y$ and $g\colon Y\to Z$ we have
		\begin{enumerate}
			\item $T_{g\circ f} = T_{g}\circ T_f$, 
			\item $N_{g\circ f} = N_g\circ N_f$,
			\item $R_{g\circ f} = R_{f}\circ R_{g}$.
			
		\end{enumerate}
	\end{lem}

	\begin{lem} \label{tambaraDoubleCosetFormula}
		For any pullback square in $\Set^G$
		\[
			\begin{tikzcd}
				A \ar["h"]{d} \ar["k"]{r} & B \ar["f"]{d} \\
				C \ar["g"]{r} & D
			\end{tikzcd}
		\]
		we have both $R_f\circ N_g = N_{k}\circ R_{h}$ and $R_f\circ T_g = T_{k}\circ R_{h}$.
	\end{lem}
	\begin{remark} \label{rem:underlyingMackey}
		The composition law $R_f\circ T_g = T_{k}\circ R_{h}$ in Lemma \ref{tambaraDoubleCosetFormula} is exactly the same composition law satisfied by restrictions and transfers in the Burnside category.  Indeed, there is a functor $i_{add}\colon \burnside^G\to \Tburnside^G$ that is the identity on objects and acts on morphisms by
		\[
			[X\xleftarrow{r} A\xrightarrow{t} Y] \mapsto  [X\xleftarrow{r} A = A\xrightarrow{t} Y].
		\]
		Restricting along $i_{add}$ gives a forgetful functor $i_{add}^*\colon \Tamb^G\to \Mack ^G$ analogous to sending a commutative ring to its underlying abelian group.
		
		Since the norms and restrictions satisfy the same interchange law as transfers and restrictions there is a second functor $i_{mult}\colon \burnside^G\to \Tburnside^G$ which is the identity on objects and acts on morphisms by 
		\[
				[X\xleftarrow{r} A\xrightarrow{t} Y] \mapsto  [X\xleftarrow{r} A \xrightarrow{t} Y= Y].
		\]
		Restricting along this functor gives a functor from the category of Tambara functors to the category of semi-Mackey functors analogous to sending a commutative ring to its underlying multiplicative monoid.
	\end{remark}

	It remains to explain how the transfers and norm maps interchange.  To do so, we first need to describe the \emph{exponential diagrams}.  For $X\in \Set^G$ we denote the category of $G$-sets over $X$ by $\Set^G_{/X}$. Associated to an equivariant map $f\colon X\to Y$, we have a functor $f^*\colon\Set^G_{/Y}\to \Set^G_{/X}$ given by pullback.  By the adjoint functor theorem, $f^*$ has a right adjoint, called the dependent product, which we denote by $f_{*}\colon \Set^G_{/X} \to \Set^G_{/Y}$.  
	
	\begin{example}\label{ex:coinductionAsDependentProduct}
		Let $n\colon G/H\to G/G$ be the collapse map.  There are equivalences of categories $\Set^G_{/(G/H)}\cong \Set^H$ and $\Set^G_{/(G/G)}\cong \Set^G$ and, passing through these equivalences, one can show that $n^*\colon \Set^G\to \Set^H$ is the usual restriction functor.  By the uniqueness of adjoints, it follows that $n_*\colon \Set^H\to \Set^G$ is the coinduction functor. 
	\end{example}
	
	Given two composable maps $A\xrightarrow{t}B\xrightarrow{n}C$ in $\Set^G$, we can form the associated \emph{exponential diagram} 
	\[
		\begin{tikzcd}
			B \ar["n"]{d}&  A  \ar["t"]{l}&  B\times_Cn_*(A) \ar["r'"]{l} \ar["n'"]{d}\\
			C & & n_*(A) \ar["t'"]{ll}
		\end{tikzcd}
	\]
	in which $n'$ is the projection and $t'$ is the map realizing $n_*(A)$ as an object in $\Set^G_{/C}$.  Noting that $B\times_{C}n_*(A) = n^{*}n_{*}(A)$, the map $r'\colon n^{*}n_*(A)\to A$ is the counit of the adjunction.  
	\begin{proposition}
		Given composable maps $A\xrightarrow{t}B\xrightarrow{n}C$ in $\Set^G$, the norm $N_n$ and transfer $T_t$ in $\mathcal{P}^G$ interchange as $N_{n}T_t = T_{t'}N_{n'}R_{r'}$ where $t'$, $n'$, and $r'$ are the maps in the exponential diagram. 
	\end{proposition}

	\begin{example} \label{c2Example}
		Let $G=C_2= \{e,\sigma\}$ be the cyclic group of order $2$, let $S\in \Tamb^{C_2}$ and let $a,b\in S(C_2/e)$.  Since addition is the transfer associated to the fold maps, the exponential diagram can be used to compute 
		\[
			N_{e}^{C_2}(a+b) = N_e^{C_2}(a)+N_e^{C_2}(b) + T_e^{C_2}(a\cdot c_{\sigma}(b))
		\]     
		For a reference on how to actually compute exponential diagrams, see Section 7 of \cite{strickland2012tambara}.
	\end{example}

	We now turn to the problem of defining bi-incomplete Tambara functors.  Suppose we are given a pair of indexing categories $\compPair$.  The idea is to take a wide subcategory $\Tburnside^G_{\mathcal{O}_m,\mathcal{O}_a}\subset \Tburnside^G$ with the transfers restricted to maps in $\mathcal{O}_a$ and the norms restricted to $\mathcal{O}_m$.  Due to the complicated composition laws of $\Tburnside^G$, some care has to be taken in which indexing categories we can choose and still have $\bitburnside$ be a category.  
	
	\begin{example}\label{ex:normOfAddition}
		Let $G=C_2$ be the group with two elements and suppose $\compPair = (\mathcal{O}^{gen},\mathcal{O}^{tr})$.  It follows from Example \ref{c2Example} that $\bitburnside$ as described in the above paragraph is \emph{not} a category.  Indeed, if $\nabla\colon C_2/e\amalg C_2/e\to C_2/e$ is the fold map, the example shows that defining the composition $N_e^{C_2}\circ T_{\nabla}$ in $\bitburnside$  requires $T_{e}^{C_2}$. Since $C_2/e\to C_2/C_2$ is not a morphism in $\mathcal{O}^{tr}$ this means the composition in $\Tburnside_{\compPair}$ would not be defined.   
	\end{example}

	The necessary condition to put on a pair of indexing categories $\compPair$ is studied in \cite{blumberghillbiincomplete} and leads to the following definition.
	\begin{defn}\label{defn:compatibleIndexingCategories}
		A pair of indexing categories $\compPair$ is \emph{compatible} if for all maps $n\colon S\to T$ in $\mathcal{O}_m$, we have 
		\[
			n_*((\mathcal{O}_a)_{/S})\subset (\mathcal{O}_a)_{/T}
		\]
		where $n_*$ is the dependent product functor.
	\end{defn}

	\begin{thm}[Theorem 3.5 of \cite{blumberghillbiincomplete}]
		If $\compPair$ is a compatible pair of indexing categories then $\bitburnside\subset \Tburnside$ is a wide  subcategory with morphisms $\bitburnside(X,Y)$ given by isomorphism classes of bispans
		\[
			[X \xleftarrow{r} A \xrightarrow{n} B \xrightarrow{t} Y]
		\] 
		with $n\in \mathcal{O}_m$ and $t\in \mathcal{O}_a$.
	\end{thm}

	\begin{defn}
		A semi $\compPair$-Tambara functor is a product preserving functor $S\colon \bitburnside\to \Set$.  A $\compPair$-Tambara functor is a semi $\compPair$-Tambara functor $S$ such that $S(X)$ is group complete for all $G$-sets $X$. When the specific indexing categories are unimportant, we refer to such objects as (semi) bi-incomplete Tambara functors.  The category of $\compPair$-Tambara functors is denoted $\Tamb^G_{\compPair}$.
	\end{defn}

	The main example of interest comes to us from topology.
	\begin{defn}
		A $G$-\emph{universe} $\mathcal{U}$ for a finite group $G$ is a real orthogonal representation of $G$ such that $\mathcal{U}^G\neq 0$ and for any finite dimensional subrepresentation $V\subset \mathcal{U}$ there exist $G$-equivariant embeddings $V^n\to \mathcal{U}$ for all $n$.
	\end{defn}
	
	\begin{example}
		A complete $G$-universe $\mathcal{U}$ is one such that every real orthogonal $G$-representation embeds into $\mathcal{U}$ infinitely many times.  A model for $\mathcal{U}$ is given by taking the infinite direct sum of the real regular representation of $G$. 
	\end{example} 

	\begin{example}
		A trivial $G$-universe $\mathcal{U}_{tr}$ is any $G$-universe with trivial $G$-action.  
	\end{example}

		Universes came about as a way to index different categories of genuine $G$-spectra $\spec_{\mathcal{U}}^G$. Different universes change the associated categories of spectra in a subtle, but important, way related to the dualizability of the the orbits $G/H$.  Briefly, the suspension spectrum $G/H_+\in \spec_{\mathcal{U}}^G$ is dualizable, in fact self-dual, if and only if there is a $G$-equivariant embedding $G/H\to \mathcal{U}$. For full details see \cite{LMayS} Section II.6.
		
		If $\mathcal{U}$ is a complete $G$-universe, then every $G$-orbit has an embedding into $\mathcal{U}$ and so all orbits are self-dual.  As a result, the canonical quotient map $G/K\to G/H$ associated to subgroups $K\leq H$ induces two maps of $G$-spectra:
		\begin{align*}
			r_K^H:G/K_+&\to G/H_+\\
			t_K^H\colon G/H_+&\to G/K_+
		\end{align*}
		For any $G$-spectrum $X$ indexed on $\mathcal{U}$, the assignment
		\begin{equation}\label{eq:homotopyMackey}
			G/H\mapsto [G/H_+,X] 
		\end{equation}
		is a Mackey functor, where restriction and transfer are given by $(r_K^H)^*$ and $(t_K^H)^*$ respectively.  Here, we have used $[-,-]$ to denote the homotopy classes of maps in the equivariant stable homotopy category.  If $X$ is a $G$-$E_{\infty}$ ring spectrum then the assignment \eqref{eq:homotopyMackey} yields a Tambara functor.
		
		When the universe $\mathcal{U}$ is not complete the situation is a bit more subtle.  The universe $\mathcal{U}$ determines two operads, the Steiner operad $S_{\mathcal{U}}$ and linear isometries operad $L_{\mathcal{U}}$, which in turn determine indexing categories we denote $\mathcal{O}_{S}$ and $\mathcal{O}_L$.  The types of indexing categories arising this way are studied in \cite{rubin_detecting}.  
		\begin{proposition}[Proposition 6.16 of \cite{blumberghillbiincomplete}]
			The pair $(\mathcal{O}_{L},\mathcal{O}_S)$ is compatible.
		\end{proposition} 
	
		Correspondingly, in an incomplete model of the equivariant stable homotopy category such as those studied in \cite{BHIncompleteStable}, the assignment \eqref{eq:homotopyMackey} naturally takes values in $\mathcal{O}_S$-Mackey functors.   If $X$ is an equivariant ring spectrum in this category, we obtain a $(\mathcal{O}_{L},\mathcal{O}_S)$-Tambara functor.
		
\section{Compatibility and Transfer Systems}\label{section:transferSystems}

	Given a pair of indexing categories $\compPair$, it is not straightforward from the definitions to check whether or not the pair is compatible.  Most of the difficulty stems from problems computing the image of the dependent product functor $n_{*}\colon \Set^G_{/S}\to \Set^G_{/T}$ associated to a map $n\colon S\to T$ in $\Set^G$.  The goal of this section is to alleviate this by providing a combinatorial characterization of when two indexing categories are compatible that can be checked without computing any dependent products.
	
	In \cite{blumberghillbiincomplete}, Blumberg and Hill provide an alternative characterization of compatibility which, in light of Example \ref{ex:coinductionAsDependentProduct}, amounts to only checking the dependent products along the canonical projections $n\colon G/K\to G/H$ in $\mathcal{O}_m$. As remarked in that example, such functors are given by the coinduction functor which sends an $K$-set $X$ to the $H$-set $\map_K(H,X)$.
	
	\begin{example}\label{ex:coinductionInC2}
		Let $H=e$ and $G=C_2$, the group with two elements.  Write $*$ for both the trivial $G$-set and trivial $H$-set.  There are four elements of $\map_e(C_2,*\amalg *)$ and a straightforward computation yields
		\[
		\map_e(C_2,*\amalg * )\cong *\amalg *\amalg C_2/e
		\]
		which (essentially) gives the formula from Example \ref{ex:normOfAddition}. More generally, we have an isomorphism:
		\[
			\map_e(C_2,n\cdot[*])\cong n\cdot[*]\amalg \left(\frac{n(n-1)}{2}\right)\cdot [C_2/e]
		\]
		Readers who are more familiar with computations in Tambara functors may recognize that this equation arises in the norm in the $C_2$-Burnside Tambara functor. 
	\end{example}
	
	\begin{thm}[Theorem 5.13 of \cite{blumberghillbiincomplete}]\label{thm:compatibilityCharacterization}
		A pair of indexing categories $\compPair$ is compatible if and only if for every pair of subgroups $K\leq H$ such that $H/K$ is $\mathcal{O}_m$-admissible and every $\mathcal{O}_a$-admissible $K$-set $T$, the coinduced $H$-set $\map_K(H,T)$ is also $\mathcal{O}_a$-admissible. 
	\end{thm}  

	Unfortunately, this characterization of compatibility can still be difficult to check in practice due to the somewhat unwieldy nature of the coinduction functors.  In particular, Example \ref{ex:coinductionInC2} shows the coinduction functor does not preserve coproducts.  Because of this, one of the standard techniques of equivariant algebra, separating a $G$-set into its orbits, often does not work when considering coinduction.  In this section we provide a characterization of compatibility, Definition \ref{defn:compatibleTransferSystems} below, that is more easily checked and does not make any reference to the coinduction functors.
	
	The main result of this section is stated in terms of transfer systems, a notion we now recall.  
	\begin{defn}
		For a finite group $G$, a $G$-transfer system $\mathcal{T}$ is a partial order $\leq_{\mathcal{T}}$ on the set of subgroups of $G$ such that
		\begin{enumerate}
			\item (refines subset relation) if $K\leq_{\mathcal{T}}H$ then $K\leq H$,
			\item (closure under conjugation) if $K\leq_{\mathcal{T}} H$, then for all $g\in G$ we have $K^g \leq_{\mathcal{T}} H^g$,
			\item (closure under intersection) If $K \leq_{\mathcal{T}} H$ then for all $L\leq G$ we have $(K\cap L)\leq_{\mathcal{T}}(H\cap L)$.
		\end{enumerate}
		
	\end{defn}
 
 	\begin{notn}
 	 	If  $\mathcal{T}$ is a $G$-transfer system it is convenient to represent $\mathcal{T}$ as a directed graph with a node for each conjugacy class of subgroup $H\leq G$ and an arrow $K\to H$ if and only if $K\leq_{\mathcal{T}}H$.  Figure \ref{fig:transferSystemExamples} gives one example and one non-example of transfer systems for the cyclic group $C_4$ drawn in this notation.  
 	 	
 	 	Later, we will have two transfer systems $\mathcal{T}_1$ and $\mathcal{T}_2$ such that $\mathcal{T}_1$ refines in $\mathcal{T}_2$ in the sense that if $K\leq_{\mathcal{T}_1} H$ then $K\leq_{\mathcal{T}_2} H$.  In such cases, we write $\mathcal{T}_1\leq \mathcal{T}_2$ and represent both transfer systems as a single directed graph with two sets of arrows:  a solid arrow $K\rightarrow H$ if $K\leq_{\mathcal{T}_1}H$ and a dashed arrow $K\dashrightarrow H$ if $K\leq_{\mathcal{T}_2}H$ but  not necessarily $K\leq_{\mathcal{T}_1}H$. When we need to be clear about whether or not $K\leq_{\mathcal{T}_1}H$ we write it explicitly. An example is given in Figure \ref{fig:twoTransferSystemsInOne}.
 	\end{notn}

	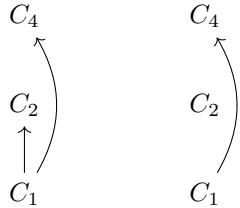
\begin{figure}[h]
		\begin{tikzcd}
			C_4 & & C_4 \\
			C_2 & & C_2 \\
			C_1\ar{u} \ar[bend right]{uu} & & C_1 \ar[bend right]{uu}
		\end{tikzcd}
		\caption{Two poset structures on the set of subsets of the cyclic group $C_4$.  The left graph is a transfer system.  The right graph is not a transfer system as it fails to satisfy closure under intersection.}
		\label{fig:transferSystemExamples}
	\end{figure}

	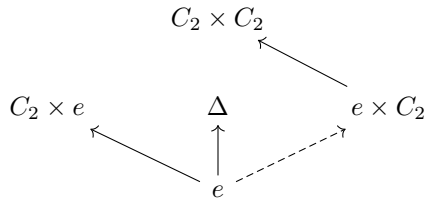
\begin{figure}[h]
		\begin{tikzcd}
			 & C_2\times C_2 & \\
			 C_2\times e & \Delta & e\times C_2 \ar{ul} \\
			 & e \ar{ul} \ar{u} \ar[dashed]{ur} &  
		\end{tikzcd}
		\caption{Two transfer systems $\mathcal{T}_1\leq \mathcal{T}_2$ on the group $C_2\times C_2$.  The group $\Delta$ is the diagonal subgroup given by the image of the diagonal map $C_2\to C_2\times C_2$.  This pair of transfer systems is not compatible.}
		\label{fig:twoTransferSystemsInOne}
	\end{figure}

	Independent work of Rubin and Balchin--Barnes--Roitzheim has shown that, given a transfer system $\mathcal{T}$, one can build an indexing category $\mathcal{O}$ generated by maps $G/K\to G/H$ such that $K\leq_{\mathcal{T}}H$.  This process can also be turned around, producing a transfer system from every indexing category, and these two constructions are mutual inverses.
	\begin{thm}[Theorem 3.7 of \cite{rubin_detecting}]
		The posets of indexing categories and transfer systems for a finite group $G$ are equivalent.
	\end{thm}

	The theorem above implies that we could just as easily have defined bi-incomplete Tambara functors in terms of the pair of transfer systems $\tCompPair$ that are equivalent to $\compPair$.  Unwinding the definitions, we see that the category $\bitburnside$ has a transfer $T_K^H$ if and only if $K\leq_{\mathcal{T}_a} H$ and a norm $N_K^H$ if and only if  $K\leq_{\mathcal{T}_m} H$.  Of course, describing $\bitburnside$ in terms of transfer systems requires that we have a notion of compatible transfer systems.
	
	\begin{defn}\label{defn:compatibleTransferSystems}
		Let $\mathcal{T}_1$ and $\mathcal{T}_2$ be two $G$-transfer systems.   We say $(\mathcal{T}_1, \mathcal{T}_2)$ is a \emph{compatible} pair if
		\begin{enumerate}
			\item $\mathcal{T}_1\leq \mathcal{T}_2$,
			\item whenever $A$ is a subgroup of $G$ and $B,C\leq A$ are subgroups such that $B\leq_{\mathcal{T}_1} A$ and $(B\cap C) \leq_{\mathcal{T}_2} B$ then $C\leq_{\mathcal{T}_2} A$. 
		\end{enumerate}
	\end{defn}

	We pause to explain the second requirement of compatibility in terms of our graphical representations of transfer systems.  The conditions that $B\leq_{\mathcal{T}_1} A$ and $(B\cap C) \leq_{\mathcal{T}_2} B$ can be represented graphically by saying the following subgraph appears in the directed graph representing these transfer systems:
	
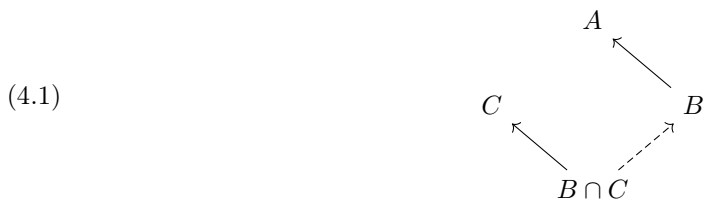
\begin{equation} \label{incompleteSquare}
		\begin{tikzcd}[column sep = small]
			 & A & \\
			 C & & B\ar{ul}\\
			 & B\cap C \ar{ul}\ar[dashed]{ur}& 
		\end{tikzcd}
	\end{equation}
	Here the solid arrow $B\cap C\rightarrow C$ exists by intersection of the solid arrow $B\rightarrow A$ with $C$. Two transfer systems $\mathcal{T}_1\leq \mathcal{T}_2$ are compatible if any subgraph of the form \eqref{incompleteSquare} is actually a part of a diamond:
	\begin{equation} \label{completeSquare}
	\begin{tikzcd}[column sep = small]
	& A & \\
	C \ar[dashed]{ur} & & B\ar{ul}\\
	& B\cap C \ar{ul}\ar[dashed]{ur}& 
	\end{tikzcd}
	\end{equation}
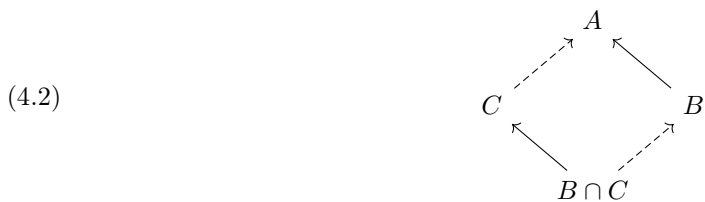
	Figure \ref{fig:A5} gives an example of a compatible pair of transfer systems $(\mathcal{T}_1,\mathcal{T}_2)$ for the alternating group $G=A_5$.  We stress that once such a diagram is drawn it is an easy matter to check that the two transfer systems are compatible, as the graphs of the transfer systems can simply be checked visually.  An example of a pair of transfer systems that are not compatible is found in Figure \ref{fig:twoTransferSystemsInOne}.
	
	\begin{figure}[h]
	\begin{tikzcd}
		& A_5                                           &                          &  &                       & A_5                       &                \\
		D_5 \arrow[ru] & A_4 \arrow[u]                                 & S_3 \arrow[lu]           &  & D_5                   & A_4                       & S_3 \arrow[lu] \\
		C_5 \arrow[u]  & C_2^2 \arrow[u]                               & C_3 \arrow[lu] \arrow[u] &  & C_5 \arrow[u, dashed] & C_2^2                     & C_3 \arrow[lu] \\
		& C_2 \arrow[u] \arrow[ruu] \arrow[luu]         &                          &  &                       & C_2 \arrow[u] \arrow[luu] &                \\
		& e \arrow[u] \arrow[u] \arrow[ruu] \arrow[luu] &                          &  &                       & e \arrow[u] \arrow[luu]   &               
	\end{tikzcd}
		\caption{An example of compatible transfer systems for the alternating group $A_5$ is drawn on the right.  The poset of conjugacy classes of subgroups of the alternating group $A_5$ is drawn on the left for reference.}
		\label{fig:A5}
	\end{figure}
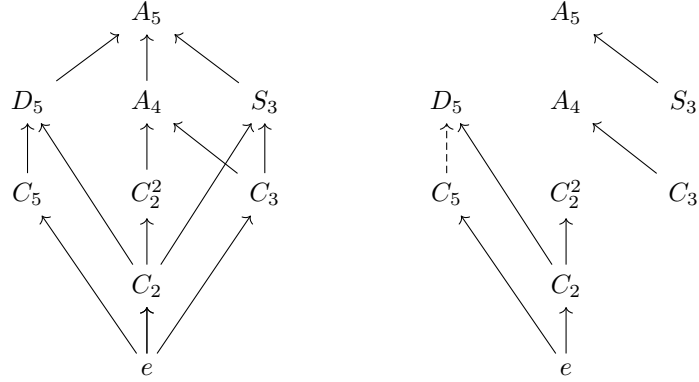

	Our definition of compatible transfer systems will only have value if we can show it is related to the notion of compatible indexing categories.  In the remainder of this section we show they are in fact equivalent.  As a first step, we have:
	\begin{proposition}\label{prop:catImpliesSystem}
		Suppose $\compPair$ is a pair of indexing categories with corresponding pair of transfer systems $\tCompPair$.  If $\compPair$ is compatible then $\tCompPair$ is as well.
	\end{proposition}

	Before proceeding with the proof, we note that the condition $\mathcal{T}_m\leq \mathcal{T}_a$ follows from Corollary 6.4 of \cite{blumberghillbiincomplete} which says we have $\mathcal{O}_m\leq \mathcal{O}_a$.  It suffices then to show the pair $\tCompPair$ meets condition (2) of Definition \ref{defn:compatibleTransferSystems}.
	
	\begin{lem}\label{lem:productCase}
		If $\compPair$ is compatible then $\tCompPair$ meets condition (2) of Definition \ref{defn:compatibleTransferSystems} in the special case where $A=BC$ is the product of the subgroups $B$ and $C$.  That is, if the product $BC=\{bc : b\in B, C\in C\}$ is a subgroup of $G$ such that $B\leq_{\mathcal{T}_m} BC$ and $(B\cap C) \leq_{\mathcal{T}_a} B$ then $C\leq_{\mathcal{T}_a} BC$. 
	\end{lem}
	\begin{proof}
		Let $A=BC$, our goal is to show $A/C$ is $\mathcal{O}_a$-admissible.  Writing $D=B\cap C$, we have, by Theorem \ref{thm:compatibilityCharacterization}, that the set $ \map_B(A,B/D)$ is $\mathcal{O}_a$-admissible.  Since indexing categories are closed under subobjects (Lemma \ref{closedUnderSubobjects}), we are done if we can show there is an element $\mu\in \map_B(A,B/D)$ whose stabilizer $A_{\mu}\subset A$ is exactly $C$ as this implies
		\[
			A/A_{\mu} =(BC)/C\subset \map_B(A,B/D)
		\]
		is $\mathcal{O}_a$ admissible.

		Given any element $x\in A$ we can write $x=bc$ for some $c\in C$ and $b\in B$.  We define the map $\mu\colon A\to B/D$ by 
		\[
		\mu(x) = \mu(bc) = bD
		\]
		Since $B\cap C =D$, the element $b$ is unique up to  right multiplication by an element of $D$ and so this map does not depend on the choice of $b$.  Since the element $c$ has no bearing on the value of $\mu$ we have $\mu(xc') = \mu(x)$ for any $c'\in C$ and thus $C\subset A_{\mu}$.  
		
		To show the reverse inclusion, take any $x\notin C$ and, as before, write $x=bc$.  Since $x\notin C$, we have $b\notin D$ and thus
		\[
			(x\cdot \mu)(e) = \mu(x) = bD\neq eD = \mu(e)
		\]
		and so $x\notin A_{\mu}$.  It follows that $A_{\mu}\subset C$ proving that $A_{\mu}=C$.
	\end{proof}

\begin{proof}[Proof of Proposition \ref{prop:catImpliesSystem}]
	
	By hypothesis we have a subgraph
	
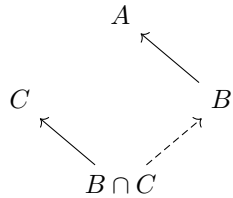
\[
	\begin{tikzcd}[column sep = small]
	& A & \\
	C & & B\ar{ul}\\
	& B\cap C \ar{ul}\ar[dashed]{ur}& 
	\end{tikzcd}
	\]
	in our transfer systems $\tCompPair$ and our goal is to show there is a dashed arrow $C\dashrightarrow A$.
	
	We proceed by induction on the index $k =[A:C]$, where the base case is trivial.  Supposing that $k>1$, let $n$ be the number of double cosets $B\backslash A/ C$ and write $\underline{n} = \{1,2,\dots, n\}$ for the trivial $B$-set with $n$ elements.  Note that $\underline{n}$ is admissible in all indexing categories by closure under finite coproducts.  We may assume $n>1$ as $n=1$ occurs only if $A=BC$, which was covered in Lemma \ref{lem:productCase}. 
	
	By Theorem \ref{thm:compatibilityCharacterization}, the set $ \map_B(A,\underline{n})$ is $\mathcal{O}_a$-admissible. Our goal is produce an element $\mu\in \map_B(A,\underline{n})$ so that $C\subset A_{\mu}\subsetneq A$.  If we can produce such a $\mu$, we have the following subgraph in our transfer systems
	
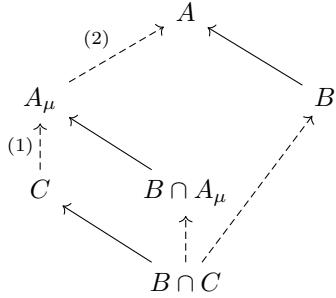
\[
	\begin{tikzcd}
		& A &  \\
		A_{\mu} \ar[dashed, "(2)"]{ur}& &  B \ar{ul}\\
		C \ar[dashed,"(1)"]{u} &B\cap A_{\mu} \ar{ul} &    \\
		& B\cap C \ar{ul} \ar[dashed]{u} \ar[ dashed]{uur}&  
	\end{tikzcd}
	\]
	
	 where the arrow labeled $(1)$ exists by induction hypothesis as the index $[A_{\mu}:C]$ is less than $k$, the arrow labeled $(2)$ exists since $A_{\mu}$ is the stabilizer of an object in $M$, which is an $\mathcal{O}_a$-admissible $A$-set, and all other arrows exist either by hypothesis or by the closure under intersection property of transfer systems.   Composing arrows $(1)$ and $(2)$ yields the result.
	
	To define $\mu$, let $a_1,a_2, \dots, a_n$ be representatives for the double cosets $B\backslash A/C$.  For simplicity, take $a_1=e$ to be the unit.  For any $x\in A$, there is a unique $i_x\in \underline{n}$ so that $x\in Ba_{i_x}C$ and we define $\mu(x) = i_x$.  Since the double cosets are fixed under left multiplication by $B$ we see that $\mu$ is $B$-equivariant.  Similarly, as the double cosets are fixed under right multiplication by $C$ we have $C\subset A_{\mu}$.  Finally, we note the containment $A_{\mu}\subset A$ is strict as
	\[
	(a_2\cdot \mu)(e) = \mu(a_2) = 2\neq 1 = \mu(e)
	\]
	so $a_2\cdot \mu\neq \mu$.  We have used here the assumption that $n>1$.
\end{proof}

To show the reverse direction, that compatible transfer systems yield compatible indexing categories, we prove a stronger result. This leads to an affirmative answer to Conjecture 7.4 of \cite{blumberghillbiincomplete}.

\begin{proposition} \label{prop:systemImpliesCat}
	Suppose that $\tCompPair$ is a compatible pair of transfer systems with corresponding indexing categories $\compPair$.  If $H/K$ is an $\mathcal{O}_m$-admissible $H$-set and $f\colon S\to T$ is a map in $i^{*}_K\mathcal{O}_a$ then $\map_K(H,f)$ is a map in $i^*_{H}\mathcal{O}_a$.  In particular, taking $T=K/K$, we see that $\map_{K}(H,S)$ is an $\mathcal{O}_a$-admissible $H$-set for all $\mathcal{O}_a$-admissible $K$-sets $S$, so $\compPair$ is compatible.
\end{proposition}
\begin{proof}
	
	To show $\map_K(H,f)$ is a map in $i^*_{H}\mathcal{O}_a$, we must show for all $\alpha\in \map_{K}(H,S)$ that $H_{f\circ\alpha}/H_{\alpha}$ is an $\mathcal{O}_a$-admissible set.  Shifting to transfer systems, we must show $H_{\alpha}\leq_{\mathcal{T}_a} H_{f\circ \alpha}$.
	
	Writing $C = \core_H(K)$, we claim that we have the following subgraph in the directed graph for the pair $\tCompPair$.
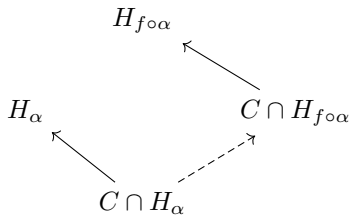
	\[
	\begin{tikzcd}[column sep = small]
		 & H_{f\circ \alpha} & \\
		 H_{\alpha} & & C\cap  H_{f\circ \alpha}\ar{ul} \\
		 & C \cap H_{\alpha} \ar{ul} \ar[dashed]{ur}&  
	\end{tikzcd}
	\] 
	
	The solid arrows exist by Lemma \ref{coreLemma} and closure under intersection. If we can establish the existence of the dashed arrow, we can use compatibility of our transfer systems to complete the diagram by filling in the last side of the diamond with a dashed arrow, completing the proof.  
	
	We claim that
	\[
		C\cap H_{\alpha} = C\cap \left(\bigcap\limits_{h\in H} h^{-1}K_{\alpha(h)}h\right) 
	\]
	where $K_{\alpha(h)}$ is the stabilizer of $\alpha(h)$ in $K$.  To prove the claim, we first note that for all $x\in C$ we have
	\[
		(x\cdot \alpha)(h) = \alpha(hx) = \alpha(hxh^{-1}h) = (hxh^{-1})\alpha(h)
	\]
	where the last equality follows from the fact that $C$ is a normal subgroup of $H$.  Thus $x\in C\cap H_{\alpha}$ if and only if for all $h\in H$ we have $\alpha(h) = (hxh^{-1})\alpha(h)$ if and only if $x\in h^{-1}K_{\alpha(h)}h$ for all $h$, establishing the claim.
	
 	The same argument shows 
	\[
	 C\cap H_{f\circ \alpha}= C\cap \left(\bigcap\limits_{h\in H} h^{-1}K_{(f\circ\alpha)(h)}h\right).
	\]
	
	Since $f\colon S\to T$ is in $i_{K}^{*}\mathcal{O}_a$ we have, by Lemma \ref{lem:stabilizerQuotients}, that $K_{(f\circ\alpha)(h)}/K_{\alpha(h)}$ is $\mathcal{O}_a$-admissible for any $h$.   The $\mathcal{O}_a$-admissibility of $(C\cap H_{f\circ\alpha})/(C\cap H_{\alpha})$ now follows from Lemma \ref{intersectionLemma}.
\end{proof}

Together, Propositions \ref{prop:catImpliesSystem} and \ref{prop:systemImpliesCat} justify the claim that compatibility for indexing categories and transfer systems are the same.

\begin{thm}\label{thm:catEqualSystem}
	Suppose $\compPair$ is a pair of indexing categories.  If $\tCompPair$ is the corresponding pair of transfer systems then $\compPair$ is compatible, in the sense of Definition \ref{defn:compatibleIndexingCategories}, if and only if $\tCompPair$ is compatible, in the sense of Definition \ref{defn:compatibleTransferSystems}.
\end{thm}

Finally, we end this section with the following corollary, affirmatively answering a conjecture of Blumberg and Hill. This corollary overcomes a technical obstacle in defining norms of incomplete Mackey functors.  After unwinding definitions, the proof is immediate from Proposition \ref{prop:systemImpliesCat}.
	\begin{cor}[Conjecture 7.4 of \cite{blumberghillbiincomplete}] \label{cor:firstConjecture}
		If $\compPair$ is a compatible pair of indexing categories then for every $\mathcal{O}_m$-admissible $H$-orbit $H/K$, coinduction restricts to a functor 
		\[
			\map_K(H,-)\colon i^{*}_K\mathcal{O}_a \to i^{*}_H\mathcal{O}_a,
		\]
		where $i_K^*\mathcal{O}$ and $i_H^*\mathcal{O}$ are the indexing categories of Example \ref{ex:restrictedIndexingCat}.
	\end{cor}

\section{Symmetric Monoidal Mackey Functors}\label{section:GCommMonoids}

In this section we recall the notion of $G$-symmetric monoidal Mackey functors in the sense of Hill and Hopkins \cite{HiHo-ESMS} and lay the groundwork for showing that bi-incomplete Tambara functors are precisely the $\mathcal{O}_m$-commutative monoids in $\mathcal{O}_a$-Mackey functors.  This result is analogous to the fact that commutative rings are the commutative monoids in the category of abelian groups. While this is a helpful analogy to have in mind, we stress that Tambara functors are \emph{not} the commutative monoids in the monoidal category of Mackey functors.  While Mackey functors do have a symmetric monoidal product, called the box product, the commutative monoids are the commutative Green functors. 

A crucial insight of Hill and Hopkins is that recovering Tambara functors from the category of Mackey functors requires considering a form of equivariant symmetric monoidal structure on the category of Mackey functors. Just as a symmetric monoidal category is a form of monoid object in categories, an equivariant symmetric monoidal structure should be a form of equivariant monoid object, \ie a semi-Mackey functor, in categories.  To keep our presentation brief, we limit our presentation to only what we need, referring the interested reader to \cite{HiHo-ESMS} for further discussion.

\begin{defn}\label{defn:symMack}
	A \emph{symmetric monoidal $\mathcal{O}$-Mackey functor} $\mathcal{C}$ consists of symmetric monoidal categories $\mathcal{C}(H)$ for every subgroup $H\leq G$, together with strong monoidal functors 
	\[
	 	R^H_K\colon\mathcal{C}(H)\to \mathcal{C}(K), \quad c_g\colon \mathcal{C}(H)\to \mathcal{C}(H^g)
	\]
	
	for any pair $K\leq H$, and for any $g\in G$.  If $H/K$ is an $\mathcal{O}$-admissible $H$-set, we also have a strong monoidal functor $N_K^H\colon C(K)\to C(H)$.  The functors $N^H_K$, $R^H_K$ and $c_g$ are called the norm, restriction, and conjugation respectively.  We require further the following coherence data:
	
	\begin{enumerate}
	\item isomorphisms $N_K^HN_L^K\cong N_L^H$, $R^K_LR^H_K\cong R^H_L$, and $c_gc_h\cong c_{gh}$ whenever these makes sense,
	
	\item For any $g\in G$, isomorphisms $c_{g}N_K^H\cong N_{K^g}^{H^g}c_g$ and $c_gR^H_K\cong R^{H^g}_{K^g}$ \label{eq:conjugationIsomorphisms}
	
	\item For subgroups $K,L\leq H$, with $H/L$ $\mathcal{O}$-admissible, the norms and restrictions are required to have a natural isomorphism
	\begin{equation}\label{eq:doubleCosetFormula}
		R^H_KN^H_L \cong \bigotimes\limits_{\gamma_i} N_{K\cap L^{\gamma_i}}^K R^{L^{\gamma_i}}_{K\cap L^{\gamma_i}} c_{\gamma_i}
	\end{equation}
	where the product on the right is the symmetric monoidal product in $\mathcal{C}(K)$ indexed over a collection of double coset representatives for $K\backslash H/L$.  We refer to this isomorphism as the double coset formula.
	\end{enumerate}
\end{defn}
\begin{remark}
	In the statement of the double coset formula it may appear we have failed to assume the $\mathcal{O}$-admissibility of the sets $K/K\cap L^{\gamma_i}$ to ensure the right side is well defined.  In fact, the admissibility of these sets follows from the admissibility of $H/L$, and closure of admissible sets under conjugation and intersection. We again stress that, just as in Remark \ref{rem:mackeyDoubleCoset}, it may occur that the right side of \eqref{eq:doubleCosetFormula} may be defined without the norm $N^H_L$ being well defined.
\end{remark}
\begin{example}
	For any $G$-indexing category $\mathcal{O}$, there is a symmetric monoidal $\mathcal{O}$-Mackey functor $\Set^\mathcal{O}$ given by $\Set^{\mathcal{O}}(H) = \Set^H$, where $\Set^H$ is the category of $H$-sets with the cocartesian monoidal structure.  The norm maps $N_K^H$ are given by the induction functors $H\times_K(-)\colon \Set^K\to \Set^H$.  The restrictions are the usual forgetful functors, and the conjugations are given by the usual isomorphism of categories $c_g\colon\Set^H\cong \Set^{H^g}$.
	
	If we instead give the categories $\Set^H$ the cartesian monoidal structure, we can construct a different symmetric monoidal $\mathcal{O}$-Mackey functor by taking the same restrictions and conjugations, and taking the norms to be coinduction functors.  The change in norms is necessitated by the fact that the double coset formula depends on the symmetric monoidal structures on the categories $\Set^H$.
\end{example}

\begin{remark}
	The double coset formula is an exact analogue the formula for ordinary Mackey functors from Lemma \ref{lem:propertiesOfMackeyFunctors} (3).  For Mackey functors, the double coset formula comes from the composition laws for the Burnside category.  In theory, a symmetric monoidal Mackey functor could be defined as a $2$-product preserving pseudofunctor out of an appropriate bicategorical Burnside category.  A similar construction has been carried out by Balmer and Dell'Ambrogio in what they call Mackey 2-functors\cite{BalmerDell}.   We have avoided doing so here because the $2$-category theory would take us beyond the scope of this paper.
\end{remark}

With symmetric monoidal Mackey functors as our model for equivariant symmetric monoidal categories, we turn to task of defining equivariant monoids in these categories.  To motivate the definitions, we consider the example of equivariant stable homotopy in which our equivariant monoids need to correspond to ring spectra.

	For $\mathcal{O}=\Set^G$, there is a symmetric monoidal Mackey functor given by $G/H\mapsto \spec^H$, the category of genuine $H$-spectra.  The restriction functors are the usual restrictions $\mathrm{Res}^H_K\colon \spec^H\to \spec^K$ and the norms $N^H_K\colon \spec^K\to \spec^H$ are given by the Hill--Hopkins--Ravenel norm \cite{HHR}.  For any $H\leq G$, denote the category of commutative ring spectra in $\spec^H$ by $\mathrm{Comm}^H$. 

\begin{proposition}[Proposition 2.27 of \cite{HHR}]	\label{prop:HHRNormAdjunction}
The norm and restriction functors restrict to an adjunction
	\[
		\begin{tikzcd}
			\mathrm{Comm}^K \ar[shift left,"N^H_K"]{r} & \mathrm{Comm}^H \ar[shift left, "R^H_K"]{l}
		\end{tikzcd}
	\]
with $N^H_K$ as the left adjoint.
\end{proposition}

Proposition \ref{prop:HHRNormAdjunction} says that a $G$-ring spectrum $E$, in addition to being a commutative monoid in $\spec^G$, comes equipped with counit maps $\mu^G_H\colon N^G_HR^G_HE\to E$, called norm multiplications by \cite{AngeltveitBohmann}, for any subgroup $H\leq G$.  Similarly, every $H$-ring spectrum $F$ comes equipped with unit maps $\eta^G_H\colon F\to R^G_HN^G_HF$.  The Hill--Hopkins model for equivariant commutative monoids is essentially just ordinary commutative monoids, together with coherent collections of unit and counit maps $\mu^G_H$ and $\eta^G_H$.  We begin by discussing the unit maps, which it turns out exist for any ordinary commutative monoid. 

For any subgroup $H\leq G$, the double coset representatives $\gamma_1,\dots, \gamma_n$ for $H\backslash G/H$ can be picked so that $\gamma_1=e$ is the identity of $G$.  This leads to a decomposition in the double coset formula:
\begin{equation}\label{eq:doubleCosetDecomposition}
	R^G_HN^G_H\cong \id_{\mathcal{C}(H)}\otimes \left( \bigotimes\limits_{i>1} N_{K\cap H^{\gamma_i}}^K R^{H^{\gamma_i}}_{K\cap H^{\gamma_i}} c_{\gamma_i}\right)	
\end{equation}

Suppose $x$ is a monoid in the symmetric monoidal category $\mathcal{C}(H)$ and write $\eta_x\colon 1_H\to x$ for the unit.  Since the restrictions, norms, and conjugation functors are all strong monoidal, they all preserve monoids and so there are unit maps $\eta_i\colon 1_H\to N_{K\cap H^{\gamma_i}}^K R^{H^{\gamma_i}}_{K\cap H^{\gamma_i}} c_{\gamma_i}(x)$ for every $i$.  In light of \eqref{eq:doubleCosetDecomposition},  we can define a map $\eta^G_H\colon x\to R^G_HN^G_H(x)$ by
\begin{equation}\label{eq:unit}
	x\cong x\otimes\left(\bigotimes\limits_{i>1} 1_H \right) \xrightarrow{1\otimes \left(\bigotimes\eta_i\right)} x\otimes \left( \bigotimes\limits_{\gamma_i, i>1} N_{K\cap H^{\gamma_i}}^K R^{H^{\gamma_i}}_{K\cap H^{\gamma_i}} c_{\gamma_i}(x)\right) \cong R^G_HN^G_H(x) 
\end{equation}

The map $\eta^G_H$ function as the unit maps for our equivariant commutative monoids.  While these maps exist for any commutative monoid, the counits represent non-trivial data which differentiates equivariant commutative monoids from the ordinary variety. 

\begin{defn}[See \cite{Hoyer} Lemma 2.7.3]\label{defn:coherence}
	For $\mathcal{C}$ a symmetric monoidal $\mathcal{O}$-Mackey functor, an \emph{$\mathcal{O}$-commutative monoid} is a monoid (in the usual sense) $x\in \mathcal{C}(G)$ together with \emph{norm multiplications} $\mu^H_K\colon N^H_KR^G_K(c)\to R^G_H(x)$ for every pair of subgroups $K\leq H$.  When the groups are clear from context, we may simply write $\mu$ for instead of $\mu^H_K$.  The norm multiplications are subject to the following coherence data:
	\begin{enumerate}
		\item For any $H\leq G$ we have the triangle identity
		\[
		\begin{tikzcd}
			R^G_H(x)\ar["{\eta^G_H}"]{r} \ar[equal]{dr}& R^G_HN_H^GR^G_H(x)\ar["{R^G_H(\mu^G_H)}"]{d}\\
			& R^G_H(x)
		\end{tikzcd}
		\]
		where $\eta^G_H$ is the map \eqref{eq:unit}.
		\item For all $L\leq K\leq H$, the map $\mu_L^H$ is equal to
		\[
		N^H_LR^G_L(x)\cong N^H_KN^K_LR^G_L(x)\xrightarrow{N^H_K(\mu^K_L)} N^H_KR^G_K(x)\xrightarrow{\mu^H_K} R^G_H(x) 
		\]
		\item For any $K\leq H$ and $g\in G$ we have a commuting square
		\[
		\begin{tikzcd}
			c_{g} N_K^HR^G_K(x) \ar["c_g\mu_K^H"]{r} \ar["\cong"]{d} & c_gR^G_H(x) \ar["\cong"]{d} \\
			N_{K^g}^{H^g}R^G_{K^g}(x) \ar["\mu^{H^g}_{K^g}"]{r} & R^G_{H^g}(x)
		\end{tikzcd}
		\]
		where the vertical isomorphisms come from the structural isomorphisms \eqref{eq:conjugationIsomorphisms} of Definition \ref{defn:symMack} 
		
		\item (compatibility with the double coset formula).  Let $L$, $K$, and $H$ be subgroups of $G$ with $L\leq K$. Fix representatives $\alpha_1,\dots,\alpha_n$ and $\beta_1,\dots \beta_m$  for the double cosets $L\backslash G/H$ and $K\backslash G/H$ respectively. The following diagram must commute:
		\[
			\begin{tikzcd}[row sep = large, column sep = huge]
				  \bigotimes\limits_{\alpha_i} N^K_{K\cap H^{\alpha_i}}N_{L\cap H^{\alpha_i}}^{K\cap H^{\gamma_i}} R^G_{L\cap H^{\alpha_i}}M  \ar["\otimes N^K_L(\mu^{K\cap H^{\alpha_i}}_{L\cap H^{\alpha_i}})"]{rr} &  & \bigotimes\limits_{\alpha_i} N^K_{K\cap H^{\alpha_i}}R^G_{K\cap H^{\alpha_i}}M \ar["F"]{d} \\
				N^K_LR^G_LN^G_HR^G_HM \ar["N^K_LR^G_L(\mu_H^G)"]{d} \ar["\cong"]{u}& & \bigotimes\limits_{\beta_i} N^K_{K\cap H^{\beta_i}}R^G_{K\cap H^{\beta_i}}M  \ar["\otimes \mu^K_{K\cap H^{\beta_i}}"]{d}\\
				N^K_LR^G_L(M) \ar["\mu_L^K"]{rr}  & & R^G_KM
			\end{tikzcd}
		\]
		The isomorphism comes from the double coset formula applied to $R^G_LN^G_H$.  To define the map $F$, note that if $\alpha_i\in K\beta_jH$, then $K\cap H^{\alpha_i}\cong K\cap H^{\beta_j}$ and so we have well defined multiplications \[\bigotimes\limits_{\alpha_i\in K\beta_jH} N^K_{K\cap H^{\alpha_i}}R^G_{K\cap H^{\alpha_i}}M\to N^K_{K\cap H^{\beta_j}}R^G_{K\cap H^{\beta_j}}M\] and $F$ is the monoidal product of all these maps.
	\end{enumerate}
\end{defn}  

\begin{defn}
	A morphism of $\mathcal{O}$-commutative monoids $x$ and $y$ is a morphism $f\colon x\to y$ of monoids in $\mathcal{C}(G)$ such that for any subgroups $K\leq H$ with $H/K$ $\mathcal{O}$-admissible, the following square commutes
	\[
	\begin{tikzcd}[row sep =large, column sep = large]
	N^H_KR^G_K(x) \ar["N^H_KR^G_K(f)"]{r} \ar["\mu^H_K"]{d} & N^H_KR^G_K(y) \ar["\mu^H_K"]{d}\\
	R^G_H(x) \ar["R^G_H(f)"]{r} & R^G_H(y)
	\end{tikzcd}
	\]
\end{defn}

	To clarify coherence condition (4) in Definition \ref{defn:coherence}, note the double coset formula gives an isomorphism:
	\[
		R^G_KN^G_HR^G_H M\cong \bigotimes\limits_{\beta_i} N^K_{K\cap H^{\beta_i}}R^G_{K\cap H^{\beta_i}}M.
	\]
	Passing through this isomorphism, and applying (4) in the degenerate case $K=L$, we can interpret the last map of (4) as 
	\[
		R^G_KN^G_HR^G_H M\xrightarrow{R^G_K(\mu^G_H)} R^G_KM
	\]
	
	In total, we have morphisms $\alpha^K_L\colon N^K_LR^G_L(N^G_HR^G_HM)\to R^G_K(N^G_HR^G_HM)$ such that the following diagram commutes:
	\[
		\begin{tikzcd}[row sep = large, column sep = large]
			N^K_LR^G_L(N^G_HR^G_HM) \ar["N^K_LR^G_L(\mu^G_H)"]{d} \ar["\alpha^K_L"]{r} & R^G_KN^G_HR^G_HM \ar["R^G_K(\mu^G_H)"]{d}\\
			N^K_LR^G)HM \ar["\mu^K_L"]{r} & R^G_KM
		\end{tikzcd}
	\]
	The above discussion indicates the proof of:
	\begin{proposition}\label{prop:inducedMonoid}
		For any $\mathcal{O}$-commutative monoid $M$ and any $H\leq G$ with $G/H$ an $\mathcal{O}$-admissible set, the object $N^G_HR^G_HM$ is also an $\mathcal{O}$-commutative monoid.  Moreover, the map $\mu^G_H\colon N^G_HR^G_HM\to M$ is a morphism of $\mathcal{O}$-commutative monoids.
	\end{proposition}

	\begin{remark}
		Proposition \ref{prop:inducedMonoid} should be thought of as a weak form of Proposition \ref{prop:HHRNormAdjunction} in our setting.  While we do not assert (yet) that the norm and restrictions give an adjunction on categories of $\mathcal{O}$-commutative monoids, we have that the map $\mu^G_H$ lives in the correct category to be a candidate for a counit map.  This provides an important technical step in our proof of the generalized Hoyer--Mazur theorem in Section \ref{section:MainTheorem}.
	\end{remark}

\section{The $\mathcal{O}_m$-symmetric monoidal Mackey functor of $\mathcal{O}_a$-Mackey functors}\label{section:biincompleteAsMonoids}

Let $\compPair$ be a compatible pair of indexing categories.  In this section we construct a symmetric monoidal $\mathcal{O}_m$-Mackey functor of $\mathcal{O}_a$-Mackey functors. This construction has previously been carried out in the complete case $\compPair = (\mathcal{O}^{gen},\mathcal{O}^{gen})$ by Mazur \cite{KMazurThesis} for $G=C_{p^n}$ and by Hoyer \cite{Hoyer} for general $G$. This section lays the groundwork for Section \ref{section:MainTheorem} where we generalize the Hoyer--Mazur theorem by characterizing $\compPair$-Tambara functors as the $\mathcal{O}_m$-commutative monoids in $\mathcal{O}_a$-Mackey functors.

Our first aim is to construct the restriction, norm, and conjugation functors which make up the structure of our $\mathcal{O}_m$-symmetric monoidal Mackey functor.  We first define analogous functors on Burnside categories with the aim being to upgrade these to operations on Mackey functor categories by left Kan extending.  The functor underlying the norm is an extension of the coinduction functor $\map_K(H,-)\colon \Set^K\to \Set^H$. A priori, it is not clear that coinduction extends to a functor of incomplete Burnside categories.  In the case of interest, our work in Section \ref{section:transferSystems} on compatible transfer systems, specifically Corollary \ref{cor:firstConjecture}, provides exactly the justification we need to give such an extension. 

A convenient consequence of defining all of our operations as left Kan extensions is that checking the coherence data of Definition \ref{defn:symMack} reduces to checking for similar coherence at the level of functors on the Burnside category. In particular, the proof that the double coset formula holds amounts to the fact that a similar formula holds on the level of sets with group action.  Since the functors on the Burnside category are very explicit, this data is easy to check by hand. 

For any subgroup $H\leq G$, we have an indexing category $i^*_H\mathcal{O}_a$ as in Example \ref{ex:restrictedIndexingCat}.  To clean up notation in this section, we denote the associated incomplete Burnside categories $\burnside_{i_H^*\mathcal{O}_a}^H$ by $\burnside^H_{\mathcal{O}_a}$. Similarly, we denote the category of $i_H^*\mathcal{O}_a$-$H$-Mackey functors by $\Mack^H_{\mathcal{O}_a}$.

For subgroups $K\leq H$, the categories $\Set^K$ and $\Set^H$ are connected by the functors
\begin{align*}
	\textrm{Res}^H_K\colon \Set^H & \to \Set^K, \\
	\map_K(H,-)\colon \Set^K & \to \Set^H.
\end{align*}

For any choices of $K$ and $H$ the functor $\textrm{Res}^H_K$ extends to a functor $\rho^H_K\colon \burnside^H_{\mathcal{O}_a}\to \burnside^K_{\mathcal{O}_a}$ because $\textrm{Res}^H_K$, being the right adjoint to induction, preserves pullback diagrams and restricts to a functor $i_H^*\mathcal{O}_a\to i_K^*\mathcal{O}_a$.  The coinduction functor does not always extend but, as noted above, there exists an extension $C_K^H\colon \burnside^K_{\mathcal{O}_a}\to \burnside^H_{\mathcal{O}_a}$ when $H/K$ is $\mathcal{O}_m$-admissible; coinduction preserves pullbacks because it is a right adjoint and it restricts to a functor on indexing categories by Corollary \ref{cor:firstConjecture}.

Similarly, for any $g\in G$ we have a conjugation isomorphism $c_g\colon \Set^H\to \Set^{H^g}$.  Abusing notation a bit, this extends to an isomorphism $c_g\colon\burnside^H_{\mathcal{O}_a}\to \burnside^{H^g}_{\mathcal{O}_a}$ of Burnside categories.

\begin{defn}\label{defn:normsAndRestrictions}
	For any $H\leq G$ and subgroup $K\leq H$ the \emph{restriction functor} $R^H_K\colon \Mack^H_{\mathcal{O}_a}\to \Mack_{\mathcal{O}_a}^K$ and \emph{conjugation functor} $c_g\colon\Mack^H_{\mathcal{O}_a}\to \Mack^{H^g}_{\mathcal{O}_a}$ are given by left Kan extension along $\rho^H_K$ and $c_g$ respectively.  If $H/K$ is $\mathcal{O}_m$-admissible, define the \emph{norm functor} $N_K^H\colon \Mack^K_{\mathcal{O}_a}\to \Mack^H_{\mathcal{O}_a}$ by left Kan extension along $C^H_K$.
\end{defn}

\begin{remark}\label{rem:colimitCommuting}
	A priori, the left Kan extension of an incomplete $K$-Mackey functor $M\colon \burnside^K_{\mathcal{O}_a} \to \Set$ along $C_K^H\colon \burnside^K_{\mathcal{O}_a}\to \burnside^H_{\mathcal{O}_a}$ is an object in the presheaf category $\mathrm{Fun}(\burnside^H,\Set)$.  That $N_K^H(M) = (C_K^H)_!M$ is actually a Mackey functor (\ie preserves products) follows from the main result of \cite{borceux_day_1977}.  While
	\[
		(C_K^H)_!\colon \mathrm{Fun}(\burnside_{\mathcal{O}_a}^K,\Set)\to \mathrm{Fun}(\burnside_{\mathcal{O}_a}^H,\Set)
	\]
	is the left adjoint of $(C_K^H)^*$, it is worth noting that $N_K^H\colon \Mack^K\to \Mack^H$ is \emph{not} a left adjoint.  In particular, the norm $N_K^H$ will usually fail to commute with colimits computed in the category of Mackey functors. Nevertheless, the norm does commute with colimits computed in the presheaf category $\mathrm{Fun}(\burnside_{\mathcal{O}_a}^K,\Set)$ which is sufficient for many purposes.
\end{remark}

Since the norm, restriction, and conjugation functors are defined via left Kan extension, we can compute their value on represented functors using the Yoneda Lemma.
\begin{lem}\label{lem:normOfRepresented}
	Let $L\leq K\leq H$ be a chain of subgroups with $H/K$ an $\mathcal{O}_m$-admissible $H$-set and let $g\in G$.  For any $K$-set $X$, let $A_{X} = \burnside_{\mathcal{O}_a}^K(X,-)$ be the represented Mackey functor of Example \ref{ex:represented}.  The norm, restriction, and conjugation of $A_X$ are can be computed as $N_K^H(A_X) \cong A_{\map_K(H,X)}$, $R^{K}_L(A_X) \cong A_{\textrm{Res}^K_L(X)}$, and $c_g(A_X) = A_{c_g(X)}$
\end{lem}

While our definition of $N^H_K$ is completely analogous to Hoyer's, our definition of the restrictions $R^H_K$ needs some justification.  To define the restrictions, Hoyer defines a functor $I_K^H\colon \burnside^K\to \burnside^H$, which is the extension of the induction functor $H\times_K(-)\colon \Set^K\to \Set^H$, and defines $R^H_K$ by precomposition with $I^H_K$. While this definition still makes sense, it is convenient to define the restrictions as a left Kan extension because it makes them easier to compare with the norm functors. For completeness, we show our definition is equivalent to Hoyer's.

\begin{proposition}\label{prop:ambidexterity}
	The functors $\rho^H_K$ and $I^H_K$ form an ambidextrous adjunction.  That is, each is both a left and right adjoint of the other.
\end{proposition}
\begin{proof}
	To clean up notation, we write simply $\rho$ and $I$.  We first show there are natural bijections:
	\[
		\burnside_{\mathcal{O}_a}^K(X,\rho Y)\leftrightarrow \burnside_{\mathcal{O}_a}^H(IX,Y)
	\]
	Going from left to right, we send a morphism $[X\xleftarrow{r} A\xrightarrow{t} \rho Y]$ to $[IX\xleftarrow{Ir}IA\xrightarrow{\hat{t}}Y]$ where the map $\hat{t}\colon IA\to Y$ is the adjunct of the map $t$ along the adjunction between induction and restriction of $H$-sets.  To see that $\hat{t}$ is a morphism in the indexing category $i_H^*\mathcal{O}_a$, consider the decomposition of $\hat{t}$ as
	\[
		IA\xrightarrow{It} I\rho Y \xrightarrow{\epsilon} Y
	\]
	where $\epsilon$ is the counit of the adjunction on $H$-sets.  The map $It$ is in $i_H^*\mathcal{O}_a$ by closure under self-induction (Lemma \ref{lem:closureUnderSelfInduction}), so it remains to show $\epsilon$ is a morphism in the indexing category.  The counit is the map
	\[
		H\times_K \mathrm{Res}^H_KY\to Y
	\] 
	which sends a class $[h,y]$ to the element $hy$.  It is an easy exercise that the stabilizer of $[h,y]$ in $H\times_K \mathrm{Res}^H_KY$ is exactly $h\textrm{Stab}_H(y)h^{-1}$, which is also the stabilizer of $hy$.  It follows that the counit is the coproduct of many fold maps, and thus is in $i_H^*\mathcal{O}_a$ by finite coproduct completeness.
	
	To build the inverse, suppose we are given a morphism $[IX\xleftarrow{p} B\xrightarrow{s} Y]$.  By Proposition 2.16 of \cite{BlumbergHillIncomplete}, the functor $I$ is an \emph{essential sieve} meaning there is a $K$-set $B'$ and a $K$-equivariant map $p'\colon B'\to X$ so that $B\cong IB'$ and $p$ factors as $B\cong IB'\xrightarrow{Ip'} IX$.  It follows that any morphism $[IX\xleftarrow{p} B\xrightarrow{s} Y]$ in $\burnside^H(IX,Y)$ is equal to one of the form $[IX\xleftarrow{Ip'} IB'\xrightarrow{s} Y]$.  We send such a morphism to $[X\xleftarrow{p'} B' \xrightarrow{\hat{s}} \rho Y]$, where again $\hat{s}$ is coming from the adjunction between restriction and induction.  Showing $\hat{s}$ is a morphism in $i_K^*\mathcal{O}_a$ is similar to the above argument for the adjunct $\hat{t}$.
	
	That the two constructions described above are inverse to one another can be understood by considering what happens to the restriction and transfer maps of $[X\xleftarrow{r} A \xrightarrow{t} \rho Y]$ separately.  For the transfer, both constructions simply replace the map by its adjunct across the $\mathrm{Ind}_K^H\dashv \mathrm{Res}_K^H$ adjunction.  For the restriction, the first construction applies $I$, and the second uses the fact that $I$ is an essential sieve to undo this.  It follows that the two maps described above are mutually inverse.  Moreover, since both the $\mathrm{Ind}_K^H\dashv \mathrm{Res}_K^H$ adjunction and the essential sieve property of $I$ are natural, this bijection is also natural in either argument establishing that $\rho$ is the right adjoint of $I$.  That this adjunction is ambidextrous is established by simply turning all the spans around are repeating the construction.
\end{proof}

\begin{cor}\label{cor:restrictionAsLeftKanExtension}
	The functor $(I_K^H)^*\colon \Mack^H\to \Mack^K$ is a model for left Kan extension along $\rho_K^H$ and thus $(I_K^H)^*\cong R^H_K$.
\end{cor}

We need to show that both $R^H_K$, $N^H_K$, and $c_g$ are strong monoidal functors.  The monoidal product on the categories $\Mack^H$ and $\Mack^K$ is the box product, defined by Day convolution \cite{Day}.  Briefly, if $M$ and $N$ are $H$-Mackey functors we define $M\squarebin N$ by the left Kan extension diagram
\[
\begin{tikzcd}
\burnside_{\mathcal{O}_a}^H\times\burnside_{\mathcal{O}_a}^H \ar["\times"']{d} \ar["M\times N"]{r} & \Set\\
\burnside_{\mathcal{O}_a}^H \ar["M\squarebin N = \mathrm{Lan}_\times (M\times N)"']{ur} &
\end{tikzcd}
\]
where $M\times N$ sends $(S,T)$ to $M(S)\times N(T)$ and the vertical map $\times$ is given by cartesian product on $H$-sets. 

\begin{lem}\label{lem:strongMonoidal}
	For any $K\leq H$, the functors $N^H_K$ (assuming $H/K$ is $\mathcal{O}_m$-admissible), $R^H_K$, and $c_g$ are all strong monoidal.
\end{lem}
\begin{proof}
	 Because the functors $\rho_K^H$, $C_{K}^H$, and $c_g$ all preserve cartesian products, they are strong monoidal functors.  The result follows from the general fact (see \cite{DayStreetKanExtensions}) that left Kan extension along strong monoidal functors is a strong monoidal.  
\end{proof}

\begin{proposition}
	With the choices of norms, restrictions, and conjugations from Definition \ref{defn:normsAndRestrictions}, the assignment $G/H\mapsto \Mack^H_{\mathcal{O}_a}$ forms a symmetric monoidal $\mathcal{O}_m$-Mackey functor.
\end{proposition}
\begin{proof}
	The restrictions norms, and conjugations are all strong monoidal by Lemma \ref{lem:strongMonoidal} so it remain to establish (1)-(3) of Definition \ref{defn:symMack}.  For subgroups $L\leq K\leq H$, there are unique natural isomorphisms $R^K_LR^H_K\cong R^H_L$,  $N^H_KN^K_L\cong N^H_L$, and $c_gc_h\cong c_{gh}$ coming from the fact that left Kan extension along a composite is isomorphic to the composite of left Kan extensions. This establishes (1), and (2) follows similarly so it remains to establish the double coset formula.
	
	For any $L,K\leq H$ and any $L$-set $X$ there is an isomorphism of $H$-sets
	\[
		\textrm{Res}^H_K\map_L(H,X) \cong \prod\limits_{\gamma_i} \map_{K\cap L^{\gamma_i}}(K, \textrm{Res}^{L^{\gamma_i}}_{K\cap L^{\gamma_i}}(c_{\gamma_i}X))
	\]
	where the $\gamma_i$ run over a transversal of the double cosets $K\backslash H/L$.  We defer the proof of this isomorphism to Lemma \ref{lem:coinductionDoubleCosetFormula} below.  For any $L\leq G$ and $L$-set $T$ write $A_{T}$ for the represented Mackey functor $\burnside^L_{\mathcal{O}_a}(T,-)$.  It is a property of Day convolution that there are natural  isomorphisms
	\[
		A_{S\times T}\cong A_S\squarebin A_T
	\] 
	for any pair of $L$-set $S$ and $T$. Using this, Lemma \ref{lem:normOfRepresented} and the set-level isomorphism above it follows that the double coset formula holds for all represented Mackey functors $A_T$.  
	
	To prove the double coset formula for an arbitrary Mackey functor $M$, we consider $M$ as an object in the presheaf category $\textrm{Fun}(A^H,\Set)$ of functors from $A^H$ to $\Set$.  This category, like all presheaf categories, is generated under colimits by the representable functors $A_T$ and so we may write
	\[
		M \cong \liminj\limits_{I} A_{T_i} 
	\]
	for some index category $I$. By Remark \ref{rem:colimitCommuting}, the norm and restriction commute with colimits in $\textrm{Fun}(A^H,\Set)$ so we have
	\begin{align*}
		R^H_KN^H_L(M) & \cong  R^H_KN^H_L\left(\liminj\limits_{I} A_{T_i}  \right)\\
		& \cong \liminj\limits_{I} R^H_KN^H_K(A_{T_i}) \\
		& \cong \liminj\limits_{I} \bigotimes\limits_{\gamma_i} N_{K\cap L^{\gamma_i}}^K R^{L^{\gamma_i}}_{K\cap L^{\gamma_i}} c_{\gamma_i}(A_{T_i})\\
		& \cong \bigotimes\limits_{\gamma_i}N_{K\cap L^{\gamma_i}}^K R^{L^{\gamma_i}}_{K\cap L^{\gamma_i}} c_{\gamma_i}\left(\liminj\limits_{I}A_{T_i}\right)\\
		& \cong \bigotimes\limits_{\gamma_i}N_{K\cap L^{\gamma_i}}^K R^{L^{\gamma_i}}_{K\cap L^{\gamma_i}} c_{\gamma_i}(M)
	\end{align*}
	where the third isomorphism uses the fact that $A_{T_i}$ is represented, and the fourth isomorphism uses the fact that Day convolution, as a left Kan extension, commutes with colimits in the presheaf category.
\end{proof}

\section{Norms and restrictions on categories of Tambara functors}\label{section:TambaraNorms}
	
	In the last section, we endowed the categories of $\mathcal{O}_a$-Mackey functors with the structure of a symmetric monoidal $\mathcal{O}_m$-Mackey functor.  We now turn our attention to characterizing the $\mathcal{O}_m$-commutative monoids in $\mathcal{O}_a$-Mackey functors.  In \cite{blumberghillbiincomplete}, it was conjectured that the  $\mathcal{O}_m$-commutative monoids are exactly the bi-incomplete Tambara functors.  In this section we lay the groundwork for proving this conjecture by studying how Tambara functors interact with the norm and restriction functors for Mackey functors.  We will show, in particular, that the norm or restriction of any Tambara functor is again a Tambara functor.
	
	It is convenient to phrase the main result of this section in slightly different language.  We construct \emph{Tambara norm functors} 
	\[
		\mathcal{N}^H_K\colon \Tamb_{\compPair}^K\to \Tamb_{\compPair}^H
	\]
	for every $\mathcal{O}_m$-admissible $H/K$ analogous to the Mackey norm functors from Definition \ref{defn:normsAndRestrictions}. In Theorem \ref{thm:normComparison} we show these two constructions agree after applying the forgetful functors from Tambara functors to Mackey functors.  Similar results when the indexing categories are complete are due to Hoyer and Mazur \cite{ Hoyer,KMazurThesis} and we adapt the proof of Theorem 2.3.3 in \cite{Hoyer}, and correct a small oversight.  Similarly, we construct \emph{Tambara restriction functors}, and also show these are compatible with the forgetful functors from Tambara functors to Mackey functors. 
	
	For any $K\leq H$, the restriction $\mathrm{Res}^H_K\colon \Set^H\to \Set^K$ and induction $H\times_K(-)\colon \Set^K\to \Set^H$ extend to functors $\rho_K^H\colon \Tburnside_{\compPair}^H  \to \Tburnside_{\compPair}^K$ and $I^H_K\colon\Tburnside_{\compPair}^K  \to \Tburnside_{\compPair}^H$ on polynomial categories.  

	\begin{defn}
		The Tambara norm functor $\mathcal{N}_K^H\colon \Tamb_{\compPair}^K\to \Tamb_{\compPair}^H$ is given by left Kan extension along the functor $I_K^H$.  The Tambara restriction functor $\mathcal{R}_K^H$ is defined by left Kan extension along $\rho^K_H$.
	\end{defn}
	\begin{remark}\label{rem:restrictionIsPullbackTambara}
		One can show that $I_K^H$ is the right adjoint of $\rho^H_K$ and it follows formally that, as in the case of Mackey functors, $\mathcal{R}^H_K$ is naturally isomorphic to the precomposition $(I_K^H)^*$.  It follows there is an adjunction $\mathcal{N}_K^H\dashv \mathcal{R}^H_K$.
	\end{remark}

Just as Lemma \ref{lem:normOfRepresented} computes the norms of represented Mackey functors, we can compute the norms of represented Tambara functors.  For any $H\leq G$ and $H$-set $T$, write $P_T$ for the represented Tambara functor $\Tburnside_{\compPair}^H(T,-)$. 
 
\begin{lem}\label{lem:normOfRepresentedTambara}
	For any $K\leq H$ and any $H$-set $T$ there is an isomorphism of Tambara functors $\mathcal{R}^H_K(P_T)\cong P_{\mathrm{Res}^H_K(T)}$.  Similarly, if $S$ is any $K$-set then $\mathcal{N}_K^H(P_S)\cong P_{H\times_K S}$.  
\end{lem}

If $i_H = i_{add}\colon \burnside_{\mathcal{O}_a}^H\to \Tburnside_{\compPair}^H$ is the inclusion functor of Remark \ref{rem:underlyingMackey}, then the forgetful functor which sends a Tambara functors to its underlying additive Mackey functor is $U_H=i_H^*$.   The main result of this section says that the Tambara norm functors ``commute'' with the forgetful functors $U_H$ in  the sense that there are isomorphisms $U_H\mathcal{N}_K^H\cong N_K^HU_K$ where $N^H_K$ is the Mackey norm functor.  Before proving this result, we quickly prove the analogous result for the Tambara restriction functors.

\begin{lem}
	For any pair $K\leq H$ of subgroups of $G$, there is a natural isomorphism of restriction functors $U_K\mathcal{R}^H_K\cong R^H_KU_H$.  
\end{lem}
\begin{proof}
	Consider the following commutative diagram of functors:
	\[
		\begin{tikzcd}
			\burnside^K_{\mathcal{O}_a} \ar["i_K"]{r} \ar["I^H_K"]{d} & \Tburnside^K_{\compPair} \ar["I^H_K"]{d}	\\
				\burnside^H_{\mathcal{O}_a}  \ar["i_H"]{r}& \Tburnside^H_{\compPair}	
		\end{tikzcd}
	\]
	
	Since this diagram commutes, there is a natural isomorphism of functors $(i_H\circ I^H_K)^*\cong (I^H_K\circ i_K)^*$.  The result now follows from Corollary  \ref{cor:restrictionAsLeftKanExtension} and Remark \ref{rem:restrictionIsPullbackTambara} which identify both the Mackey and Tambara restriction functors precomposition with $I^H_K$.
\end{proof}

Comparing the Tambara and Mackey restriction functors with the forgetful functors $U_K$ and $U_H$ is easy because all functors involved are precomposition functors.  Comparing the Tambara and Mackey norms requires more care because it compares precomposition functors with left Kan extensions.  The correct categorical framework in which to approach such comparisons is the calculus of mate diagrams, which we now recall.

Consider the following square of functors, inhabited by a natural transformation $\alpha$.
\begin{equation}\label{eq:BCsquare1}
	\begin{tikzcd}[row sep =large,  column sep = large]
	A \ar["f"]{r} \ar["h"]{d} & B \ar["k"]{d}\\
	C\ar["g"]{r} & D \ar[shorten =4mm,Rightarrow, from = 1-2, to = 2-1,"\alpha"']
	\end{tikzcd}
\end{equation}

For any category $\mathcal{C}$, we denote the category of functors and natural transformations from $\mathcal{C}$ to $\Set$ by $\Set^{\mathcal{C}}$.  The square \eqref{eq:BCsquare1} determines another square of functors:
\begin{equation}\label{eq:BCsquare2}
\begin{tikzcd}[row sep =large,  column sep = large]
	\Set^A  & \Set^B \ar["f^*"']{l}\\
	\Set^C\ar["h^*"']{u} & \Set^D \ar["k^*"']{u} \ar["g^*"']{l} \ar[shorten =4mm, Rightarrow, from = 1-2, to = 2-1,"\alpha^*"']
\end{tikzcd}
\end{equation}

Because all $\Set$ valued functors admit left Kan extensions, the functors $h^*$ and $k^*$ in \eqref{eq:BCsquare2} admit left adjoints we denote by $h_!$ and $k_!$ respectively.  We denote the units and counits of these adjunctions by $\eta_h$, $\eta_k$, $\epsilon_h$ and $\epsilon_k$ respectively.  Using the units and counits, we can define a natural transformation $\beta\colon h_!f^*\Rightarrow g^*k_!$ as the composite
\[
	h_!f^*\xRightarrow{h_!f^*\cdot \eta_k} h_!f^*k^{*}k_! \xRightarrow{h_!\cdot \alpha^*\cdot k_!} h_!h^*g^*k_!\xRightarrow{\epsilon_h\cdot g^*k_!} g^*k_!
\]
which fills the square
\begin{equation}\label{eq:BCsquare4}
	\begin{tikzcd}[row sep =large,  column sep = large]
	\Set^A  \ar["h_!"]{d}& \Set^B \ar["k_!"]{d}\ar["f^*"']{l}\\
	\Set^C & \Set^D  \ar["g^*"']{l} \ar[shorten =4mm,Rightarrow,from =1-1,to=2-2,"\beta"]
	\end{tikzcd}
\end{equation}

\begin{defn}
	The square \eqref{eq:BCsquare4} is called the \emph{mate} of the square \eqref{eq:BCsquare2}.  We say the square \eqref{eq:BCsquare1} is \emph{exact}, or satisfies the \emph{Beck--Chevalley condition}, if $\beta$ is a natural isomorphism.
\end{defn}

To place our work in the framework of mates, we construct a natural transformation $\alpha\colon i_HC^H_K\Rightarrow I^H_Ki_K$ so that the square
\begin{equation}\label{eq:squareIsBC}
\begin{tikzcd}[row sep =large,  column sep = large]
	\burnside_{\mathcal{O}_a}^K  \ar["C^H_K"]{d} \ar["i_K"]{r}& \Tburnside^K_{\compPair} \ar["I^H_K"]{d}\\
	\burnside_{\mathcal{O}_a}^H \ar["i_H"]{r}& \Tburnside^H_{\compPair}  \ar[shorten =4mm, Rightarrow, from = 1-2, to = 2-1,"\alpha"']
\end{tikzcd}
\end{equation}
is exact.  For a $K$-set $T$, the component $\alpha_T\colon I^H_Ki_K(T)\to i_HC^H_K(T)$ is represented by the bispan
\[
	H\times_{K}T\xleftarrow{H\times_{K}(\epsilon^C_{T})} H\times_K\mathrm{Res}^H_K\map_{K}(H,T) \xrightarrow{\epsilon^I_{\map_K(H,T)}}  \map_{K}(H,T) \xrightarrow{=} \map_{K}(H,T)
\]
where $\epsilon^C$ and $\epsilon^I$ are the counits of the coinduction-restriction and induction restriction adjunctions respectively. This is indeed a natural transformation, although the proof is rather involved and we defer it to Section \ref{section:technicalProofs}.

\begin{lem}\label{lem:alphaIsNatural}
	The maps $\alpha_T\colon I^H_Ki_K(T)\to i_HC^H_K(T)$ give a natural transformation $\alpha\colon i_HC^H_K\Rightarrow I^H_Ki_K$.
\end{lem}

Equipped with the natural transformation $\alpha\colon I^H_Ki_K\Rightarrow i_HC^H_K$, we obtain the mate transformation $\beta\colon (C^H_K)_!I_K^*\Rightarrow i_H^*(I^H_K)_!$ and the main result of section is that $\beta$ is a natural isomorphism.  

For any bi-incomplete Tambara functor $S\colon \Tburnside^K_{\compPair}\to \Set$ and $H$-set $Y$ the pointwise Kan extension formula allows us to write the elements of $(C^H_K)_!U_K(S)(Y)$ as equivalence classes of pairs $(\omega\colon C^H_KB\to Y,x\in R^H_K(B))$ where $\omega$ is a morphism in $\burnside^H_{\mathcal{O}_a}$.  The equivalence classes are generated by the relations
\[
	(\omega\circ C^H_K(\omega'),x)\sim (\omega,S(i_K(\omega'))(x))
\] 

for any maps $\omega'$ in $\burnside^K_{\mathcal{O}_a}$.  The component $\beta_{S,Y}\colon (C^H_K)_!i_K^*(S)(Y)\to i_H^*(I^H_K)_!(S)(Y)$  of $\beta$ sends the class represented by $(\omega\colon C^H_KB\to Y,x)$ to the class of $(i_H(\omega)\circ \alpha_{B}\colon I^H_KB\to Y,x)$.

\begin{thm}[cf. Hoyer \cite{Hoyer}, Theorem 2.3.3]\label{thm:normComparison}
	For any $\mathcal{O}_m$-admissible $H/K$, the square \ref{eq:squareIsBC} is exact.  That is, the natural transformation $\beta$ is a natural isomorphism of functors $U_H\mathcal{N}_K^H\cong N_K^HU_K$.
\end{thm}

\begin{remark}
	In proving a version of Theorem \ref{thm:normComparison}, Hoyer defines a map, which in our notation is $\beta^{-1}$, and shows it is well defined.  It appears that Hoyer's proof does not show that the inverse of his map, which we call $\beta$, is well defined.  The advantage of first defining the natural transformation $\alpha$ is that the well-definedness of $\beta$ is immediate from the fact that it is the mate of the natural transformation $\alpha$.
\end{remark}

\begin{proof}
	For notational brevity, we fix the subgroups $K$ and $H$ and suppress them from the notation when clear, writing $C$ instead of $C^H_K$, $\rho$ instead of $\rho^H_K$ and so on.  Fixing a bi-incomplete $K$-Tambara functor $S$ and an $H$-set $Y$, we need to show the components $\beta_{S,Y}\colon C_!i_K^*(S)(Y)\to i_H^*I_!(S)(Y)$ are bijections.  To show surjectivity, note that by Lemma 2.3.4 of \cite{Hoyer}, and the discussion that follows, an arbitrary element in the codomain of $\beta_{S,Y}$ is a class represented by a pair $(\omega,x)$ where $\omega$ is a bispan of the form
	\[
		\omega =[I\rho B \xleftarrow{=} I\rho B \xrightarrow{\epsilon^I} B \xrightarrow{h}Y]
	\]
	  
	One can check directly $\beta_{S,Y}([C\rho B\xleftarrow{\eta^C} B\xrightarrow{h}Y,x]) = [\omega,x]$ and thus our map is surjective.  Since the element $[\omega,x]$ is arbitrary, the assignment 
	\begin{equation}\label{eq:sectionOfBeta}
		[\omega,x]\mapsto [C\rho B\xleftarrow{\eta^C} B\xrightarrow{h}Y,x]
	\end{equation}
	defines a section of $\beta$ which Hoyer shows is well defined. 
	
	It remains to show our section \eqref{eq:sectionOfBeta} is surjective.  To see this, suppose we are given an arbitrary element $[\sigma,x]\in (C^H_K)_!i_K^*(S)(Y)$ where $\sigma$ is represented by the span $CX\xleftarrow{r} A\xrightarrow{t} Y$ and $x\in S(X)$.  We can factor the map $r$ as the composite $A \xrightarrow{\eta^C}C\rho(A) \xrightarrow{C(\hat{r})} CX$ where $\widehat{r}$ is the adjunct of $r$.  Using the defining relation of the pointwise Kan extension formula, we see that $[\sigma,x] = [\sigma', S(\hat{r})(x)]$ where $\sigma'$ is represented by the span $C\rho A\xleftarrow{\eta^C} A\xrightarrow{t} Y$.  since $[\sigma', S(\hat{r})(x)]$ is in the image of \ref{eq:sectionOfBeta}, we are done.
\end{proof}

\section{Bi-incomplete Tambara Functors are $\mathcal{O}_m$-Commutative Monoids}\label{section:MainTheorem}

In this section we prove the $\mathcal{O}_m$-commutative monoids in $\mathcal{O}_a$-Mackey functors are exactly the $\compPair$-Tambara functors. We begin by showing that for any $G$-Tambara functor $S$, the underlying Mackey functor $U_G(S)$ is always an $\mathcal{O}_m$-commutative monoid by constructing norm multiplications 
\[
	\mu^H_K\colon N^H_KR^G_KU_G(S)\to R^G_HU_G(S)
\]
for all pairs $K\leq H$ of subgroups of $G$.  The construction of the $\mu^H_K$ rely heavily on our work from Section \ref{section:TambaraNorms} comparing the norm and restriction functors with the forgetful functor $U_G$.

After establishing that Tambara functors give $\mathcal{O}_m$-commutative monoids we turn our attention to the inverse construction of constructing a Tambara functor from an  $\mathcal{O}_m$-commutative monoid $M$.  The essential step is to use the norm multiplications to build operations $\nu_K^H\colon M(G/K)\to M(G/H)$ which we call the \emph{external norms} of $M$.  We show that if $M$ is the underlying monoid of a Tambara functor then the external norms agree with the usual internal norms of the Tambara functor.  Now only does this show that two constructions are mutually inverse, but is also implies that the external norms are compatible with the transfers and restrictions in exactly the same way as the norms of a Tambara functor, allowing us to conclude that every $\mathcal{O}_m$-commutative monoid equipped withe external norms is a Tambara functor.

	\begin{proposition}\label{prop:tambaraAreMonoid}
		For any Tambara functor $S\in \Tamb_{\compPair}^G$, the Mackey functor $U_G(S)$ is an $\mathcal{O}_m$-commutative monoid.  Moreover, this gives a functor 
		\begin{equation}\label{eq:functor1}
			U_G\colon \Tamb^G_{\compPair}\to \mathrm{Comm}_{\mathcal{O}_m}(\Mack_{\mathcal{O}_a}).
		\end{equation}
	\end{proposition}
	\begin{proof}
		For any map $h\colon G/K\to G/H$ in $\mathcal{O}_m$ we define the norm multiplication 
		\[
			\mu^H_K\colon N^H_KR^G_KU_G(S)\to R^G_H U_G(S)
		\]
		to be the unique map so that the following diagram commutes:
		\begin{equation}\label{eq:normMultForTambara}
			\begin{tikzcd}[row sep = large,column sep = huge]
				N^H_KR^G_KU_G(S) \ar["\mu^H_K"]{rr} \ar["\cong"]{d} & & R^G_HU_G(S) \ar["\cong"]{d}\\
				U_G\mathcal{N}^H_K\mathcal{R}^G_KS \ar["\cong"]{r}  &U_G\mathcal{N}^H_K\mathcal{R}^H_K\mathcal{R}^G_HS \ar["U_G\cdot \epsilon^H_K "]{r}  & U_G\mathcal{R}^G_HS
			\end{tikzcd}
		\end{equation}
		where $\epsilon^H_K \colon \mathcal{N}^H_K\mathcal{R}^H_K\mathcal{R}^G_HS\to \mathcal{R}^G_HS$ is the counit of the $\mathcal{N}^H_K\dashv \mathcal{R}^H_K$ adjunction and the vertical maps are the isomorphisms of Theorem \ref{thm:normComparison}.  The coherence data of Definition \ref{defn:coherence} is checked directly using the explicit form of the isomorphism of Theorem \ref{thm:normComparison} and the fact the counits $\epsilon^H_K$ can be explicitly computed using the pointwise Kan extension formula.
		
		To see $U_G$ is a functor we must show that for any map $F\colon S\to T$ of Tambara functors that the underlying map $f = U_G(F)\colon U_G(S)\to U_G(T)$ is a morphism of $\mathcal{O}_m$-commutative monoids.  This amounts to showing that the square
		\[
		\begin{tikzcd} [row sep =large, column sep = large]
			N^H_KR^G_KU_G(S) \ar["N^H_KR^G_Kf"]{r} \ar["\mu^H_K"]{d} & N^H_KR^G_KU_G(T) \ar["\mu^H_K"]{d}\\
			R^G_HU_G(S)  \ar["R^G_Hf"]{r} & R^G_HU_G(T) 
		\end{tikzcd}
		\]
		commutes for any choices of $K\leq H$ but this is immediate the every map in the diagram \eqref{eq:normMultForTambara} is natural.
	\end{proof}

	The remainder of this section is devoted to constructing an inverse to the functor $U_G$.  Since every $\mathcal{O}_m$-commutative monoid has the structure of a Green functor, and Tambara functors are essentially Green functors with norm maps, it suffices to construct functorial norm operations on every $\mathcal{O}_m$-commutative monoid.  We call the operations we construct \emph{external norms}.
	
	Let $M$ be an $\mathcal{O}_m$-commutative monoid in $\mathcal{O}_a$-Mackey functors.  Any element $x\in M(G/K)$ determines a map $\hat{x}\colon A_{K/K}\to R^G_KM$. If $H/K$ is $\mathcal{O}_m$-admissible, we can form the following composite
	\begin{equation}\label{eq:externalNorm}
	A_{H/H}\cong N^H_KA_{K/K}\xrightarrow{N^H_K(\hat{x})} N^H_KR^G_KM\xrightarrow{\mu} R^G_HM.
	\end{equation}
	where $\mu\colon N^H_KR^G_KM\to R^G_HM$ is the norm multiplication.  By the Yoneda lemma, the map \eqref{eq:externalNorm} corresponds uniquely to an element $\nu_K^H(x)\in R^G_HM(H/H) \cong M(G/H)$ which we call the \emph{external norm} of $x$.  When $M=U_G(S)$ comes from a  Tambara functor, the external norms recover the usual internal norms. 
	
	\begin{proposition}\label{prop:externalNormIsInternalNorm}
		For any Tambara functor $S$ and element $x\in S(G/K)$ the external norm $\nu_K^H(x)$ is equal to the internal norm $N_K^H(x)$.
	\end{proposition}
	
	\begin{proof}
		Using the definition of the norm multiplications on a Tambara functor, the composite \eqref{eq:externalNorm} defining $\nu^H_K(x)$ becomes 
		\begin{equation}\label{eq:externalNormTambara}
			A_{H/H}\cong N_K^HA_{K/K} \xrightarrow{N_K^H(\hat{x})} N_K^HU_K\mathcal{R}^G_KS \xrightarrow{\beta\cdot\mathcal{R}^G_K} U_H\mathcal{N^H_K}\mathcal{R}^G_KS \xrightarrow{U_H\cdot \epsilon_{R^G_HS}} U_H\mathcal{R}^G_HS
		\end{equation}
		where $\epsilon_{R^G_HS}\colon \mathcal{N}^H_K\mathcal{R}^G_KS\to \mathcal{R}^G_HS$ is the counit of the $\mathcal{N}^H_K\dashv \mathcal{R}^H_K$ adjunction.
		
		The external norm $\nu_K^H(x)$ is equal to the image of the element $\id_{H/H}\in A_{H/H}(H/H)$ under the map \eqref{eq:externalNormTambara}.  Evaluating all the functors in \eqref{eq:externalNormTambara} at the object $H/H$, and using the pointwise Kan extension formula, we can can compute where $\id_{H/H}$ is sent at every step in the composition:
		\begin{align*}
		\id_{H/H} & \mapsto [C(K/K)\xleftarrow{\cong} H/H \xrightarrow{=}H/H,\id_{K/K}\in A_{K/K}(K/K)] \\
		& \mapsto [C(K/K)\xleftarrow{\cong} H/H \xrightarrow{=}H/H,x\in U_K\mathcal{R}^G_KS(K/K)] \\
		& \mapsto [I\rho (H/H) \xleftarrow{=} I\rho (H/H) \xrightarrow{\epsilon^I_{H/H}} H/H\xrightarrow{=}H/H, x\in \mathcal{R}^G_KS(K/K)]\\
		& \mapsto \mathcal{R}^G_H(S)(N_{\epsilon^I_{H/H}})(x).
		\end{align*}
		
		In the last line, we use the Yoneda lemma to identify $\mathcal{R}^G_KS(K/K)$ with $\mathcal{R}^G_HS(H/K)$ and call the element $x$ the same thing in both sets.  Under the identification $\mathcal{R}^G_H(S)(H/H)\leftrightarrow S(G/H)$, this element corresponds to $S(N_{I_H^G(\epsilon^I_{H/H})})(x)$.  The object $I\rho(H/H) = H\times_K(\mathrm{Res}^H_K(H/H))\cong H/K$ and, under this isomorphism, the map $\epsilon_{H/H}^I\colon H/K\to H/H$ is the canonical quotient and so we have $\nu_K^H(x) = S(N_{G/K\to G/H})(x) =N_K^H(x)$.
	\end{proof}
	
	We can define an external norm $\nu_p\colon M(S)\to M(T)$ for any map $p\colon S\to T$ in $\Set^G$.  First, if $p\colon G/K\to G/H$ is the canonical quotient we define $\nu_p=\nu_K^H$.  Any other map $p$ in $\Set^G$ is isomorphic to a disjoint union of canonical quotients
	\[
	p\colon \coprod\limits_{i=1}^n G/K_i\to G/H
	\]
	for some subgroups $K_i\leq H$.  For such $p$, we define $\nu_p$ to be the composition
	\[
	\prod\limits_{i=1}^n M(G/K_i)\xrightarrow{\prod\nu_{K_i}^{H}} \prod \limits_{i=1}^n M(G/H) \xrightarrow{\mu} M(G/H)
	\]
	where $\mu$ is the multiplication map that exists because $M$ is a Green functor. 

\begin{proposition}\label{prop:monoidsAreTambara}
	Suppose that $M\in \Mack_{\mathcal{O}_a}^G$ is an $\mathcal{O}_m$-commutative monoid.  Then there exists a $\compPair$-Tambara functor $S\colon \Tburnside^G_{\compPair}\to \Set$ whose underlying Green functor is $M$.  The internal norms of $S$ are given by the external norms $\nu^H_K$.
\end{proposition}
\begin{proof}
	
	We proceed by induction on the size of the group $G$.  The base case of the trivial group is immediate since both Tambara functors and $i_e^*\mathcal{O}_m$-commutative monoids are just commutative rings.  Suppose then that the result is true for all groups $H$ with $|H| < |G|$.
	
	By definition, the Mackey functor $M$ has the structure of a commutative monoid in $\Mack^G$, \ie a Green functor.  Since a Tambara functor is just a Green functor with additional norm maps, it suffices to show how the $\mathcal{O}_m$-commutative monoid structure determines norms $M(G/K)\to M(G/H)$ for each $K\leq H$ with $H/K$ an $\mathcal{O}_m$-admissible $H$-set. As indicated in the statement of the proposition, the norm maps are given by the external norm maps $\nu_K^H\colon M(G/K)\to M(G/H)$.

	All that remains is to check the external norms satisfy the necessary compatibility with the transfers and restriction maps of our Green functor $M$. 
	
	We first handle compatibility with transfers. Suppose we have $L\leq K\leq H$, and suppose we have picked an exponential diagram in $\Set^G$
	\[
		\begin{tikzcd}
			E \ar["\alpha"]{d} \ar["\beta"]{rr} & & F \ar["\gamma"]{d}\\
			G/L \ar["p"]{r} & G/K \ar["q"]{r} & G/H
		\end{tikzcd}
	\]
	where $p$ and $q$ are the canonical quotient maps.  We need to show that for any $x\in M(G/L)$ that $\nu_K^H(T_L^K(x)) = T_{\gamma}\nu_{\beta}R_{\alpha}(x)$.  We assume that $L$ is a proper subgroup of $H$, as otherwise there is nothing to show.
	
	We claim there is a Tambara functor $S$ and a map $f\colon S\to M$ of $\mathcal{O}_m$-commutative monoids so that $x=f(y)$ for some $y\in S(G/L)$.  Granting this, the naturality of external norms, transfers and restrictions implies
	\[
		\nu_K^H(T_L^K(x)) = \nu_K^H(T_L^K(f(y))) = f(\nu_K^H(T_L^K(y))) = f(T_{\gamma}\nu_{\beta}R_{\alpha}(y)) = T_{\gamma}\nu_{\beta}R_{\alpha}(x)
	\]
	where the third equality uses the fact that, by Proposition \ref{prop:externalNormIsInternalNorm}, the external norms in $U_GS$ must be equal to the internal norms of $S$ and thus are sufficiently compatible with the transfers. The same argument proves compatibility between the external norms and the restrictions and so it remains to prove the claim.
	
	By the induction hypothesis, and the assumption that $L\leq H$ is proper, the $i_L^*\mathcal{O}_m$-commutative monoid $R^G_L(M)$ is isomorphic to $U_L(S')$ for some $L$-bi-incomplete Tambara functor $S'$.  Applying the Tambara norm, and using Theorem \ref{thm:normComparison}, we have an equivalence of $\mathcal{O}_m$-commutative monoids $U_G\mathcal{N}^G_LS'\cong N^G_LR^G_LM$.  In the commutative square
	\[
		\begin{tikzcd}
			R^G_LN^G_LR^G_H(M)(L/L)\ar["R^G_L(\mu^G_L)"]{d} \ar["\cong"]{r} & N^G_LR^G_L(M)(G/L) \ar["\mu^G_L"]{d}\\
			R^G_L(M)(L/L)\ar["\cong"]{r} & M(G/L)
		\end{tikzcd}
	\]
	the left vertical arrow is surjective by (1) of Definition \ref{defn:coherence}.  It follows that the right arrow is as well and so $x$ is in the image of the map $U_G\mathcal{N}_L^G(S')\cong N_L^GR_L^G(M)\xrightarrow{\mu_L^G}M$.  This map is a morphism of $\mathcal{O}_m$-commutative monoids by Proposition \ref{prop:inducedMonoid} so we have proven the claim.
\end{proof}

Proposition \ref{prop:monoidsAreTambara} gives the object function of a functor 
\begin{equation}\label{eq:equivalenceFunctor}
	\Phi\colon \textrm{Comm}_{\mathcal{O}_m}(\Mack_{\mathcal{O}_a})\to \Tamb_{\compPair}^G
\end{equation}
from the category of $\mathcal{O}_m$-commutative monoids in $\mathcal{O}_a$-Mackey functors to the category of Tambara functors.  On morphisms, this functor sends a map of $\mathcal{O}_m$-commutative monoids to the map of underlying commutative monoids \ie to the map of Green functors.  The compatibility conditions on morphisms of $\mathcal{O}_m$-commutative monoids imply that such maps of Green functors commute with the external norms and hence are maps of Tambara functors.  Since two maps of Tambara functors are the same if and only if the underlying maps of Mackey functors are the same, this functor is faithful.  On the other hand, by Theorem \ref{prop:externalNormIsInternalNorm}, the composite $\Phi\circ U_G$ is actually the identity functor, and thus $\Phi$ is also full and surjective on objects.

\begin{thm}[Conjecture 7.8 of \cite{blumberghillbiincomplete}] \label{thm:secondMainTheorem}
	For any compatible pair of indexing categories $\compPair$, the functor
	\[
		\Phi\colon \mathrm{Comm}_{\mathcal{O}_m}(\Mack_{\mathcal{O}_a}) \to \Tamb_{\compPair}^G
	\]
	is an equivalence of categories.
\end{thm}

\section{Proofs of Technical Lemmas}\label{section:technicalProofs}

This section contains the proofs of two lemmas used elsewhere in the paper. We have deferred the proofs of these lemmas as the details are non-essential to understanding the goals of the paper and consist mostly of diagram chases and formal category theory.
The first lemma is a double coset formula for the coinduction functor, used in the proof that $\mathcal{O}_a$-Mackey functors form a symmetric monoidal $\mathcal{O}_m$-Mackey functor.  This formula is surely well-known, but we could not find a reference.  The second is the proof of Lemma \ref{lem:alphaIsNatural} that the maps $\alpha\colon i_HC^H_K\Rightarrow I^H_Ki_K$ actually assemble into a natural transformation.  We first give the proof of the double coset formula.

\begin{lem}\label{lem:coinductionDoubleCosetFormula}
	Let $H$ be a finite group and suppose $L$ and $K$ are two subgroups of $H$.  Let $\gamma_1,\dots,\gamma_n$ be a collection of representatives for the double cosets $K\backslash H/L$.  For any $L$-set $X$, there is an isomorphism of $K$-sets
	\[
		\textrm{Res}^H_K\map_L(H,X) \cong \prod\limits_{\gamma_i} \map_{K\cap L^{\gamma_i}}(K, \textrm{Res}^{L^{\gamma_i}}_{K\cap L^{\gamma_i}}(c_{\gamma_i}X))
	\]
	where $c_{\gamma_i}(X)$ is the $L^{\gamma_i}$-set with the same objects as $X$ but and action defined by $(\gamma_il\gamma_i^{-1})\cdot x = lx$.
\end{lem}

\begin{proof}
	For any group $G$, let $BG$ denote the category with one element and morphism set $G$.  We adopt the convention that the composite
	\[
		\bullet \xrightarrow{g_1} \bullet \xrightarrow{g_2} \bullet
	\]
	is $g_1g_2$, instead of $g_2g_1$.
	
	The category $\Set^G$ is equivalent to the category of functors $\mathrm{Fun}(BG,\Set)$.  For any $J\leq G$, the inclusion of subcategories $i_J\colon BJ\to BG$ determines the restriction $\mathrm{Res}^G_J =i_J^*\colon \Set^G\to \Set^J$.  By uniqueness of adjoints, the coinduction functor $\map_J(G,-)\colon \Set^J\to \Set^G$ is isomorphic to right Kan extension  functor $(r_J)_*$.
	
	Consider the following square of groupoids
	\[
		\begin{tikzcd}[row sep = large,column sep = large]
			\mathrm{Comma}(i_K,i_L) \ar["u"]{r} \ar["v"]{d} & BL\ar["i_L"]{d} \\
			BK \ar["i_K"]{r} & BH \ar[shorten =3mm,Rightarrow,from =2-1,to=1-2,"\phi"']
		\end{tikzcd}
	\]
	where $\mathrm{Comma}(i_K,i_L)$ is the comma category $i_K/i_L$ and $\phi$ is the canonical comma natural transformation. The objects of  $\mathrm{Comma}(i_K,i_L)$ are all the elements $h\in H$ and an arrow $h\to h'$ is a pair $(k,l)\in K\times L$ such that $khl^{-1}=h'$.  
	
	By Proposition 1.26 of \cite{GrothStable}, the comma square is exact and meaning there is a natural isomorphism $v_*u^*\cong i_K^*(i_L)_*$.  For any $L$-set $X$, we have observed that $i_K^*(i_L)_*(X)\cong \mathrm{Res}^H_K\map_L(H,X)$ is the left hand side of our claimed isomorphism.  It remains to identify $v_*u^*(X)$.
		
	From the description of morphisms in $\mathrm{Comma}(i_K,i_L)$ we see that $\pi_0\mathrm{Comma}(i_k,i_L)$ is in bijection with the double coset $K\backslash H/L$.  For any $\gamma_i$, we have $\pi_1(\mathrm{Comma}(i_K,i_L),\gamma_i)$ is equal to the set of pairs $(k,l)$ such that $k = \gamma_il\gamma_i^{-1}$, which is naturally isomorphic to $K\cap L^{\gamma_i}$. Since every groupoid is isomorphic to the union of of the fundamental groups of its components, we have an isomorphism of groupoids
	\[
		\mathrm{Comma}(i_K,i_L)\cong \coprod\limits_{\gamma_i} B(K\cap L^{\gamma_i}) 
	\]
	which gives a natural isomorphism
	\begin{equation}\label{eq:commaSquareIso}
		\mathrm{Fun}(\mathrm{Comma}(i_K,i_L),\Set)\cong \prod\limits_{\gamma_i} \mathrm{Fun}(B(K\cap L^{\gamma_i}),\Set)\cong \prod\limits_{\gamma_i} \Set^{B(K\cap L^{\gamma_i})}
	\end{equation}
	
	After identification \eqref{eq:commaSquareIso}, the map 
	\[
		v^*\colon \Set^K\to \prod\limits_{\gamma_i} \Set^{B(K\cap L^{\gamma_i})}
	\] 
	is the product of the restriction functors $\mathrm{Res}^K_{K\cap L^{\gamma_i}}$ and the map $u^*$ is the product of the conjugations  isomorphic $BL\cong B(L^{\gamma_i})$ followed by the restrictions $\mathrm{Res}^{L^{\gamma_i}}_{K\cap L^{\gamma_i}}$.  Given any $L$-set $X$, we have computed
	\[
		v_*u^*(X)\cong \prod\limits_{\gamma_i} \map_{K\cap L^{\gamma_i}}(K, \textrm{Res}^{L^{\gamma_i}}_{K\cap L^{\gamma_i}}(c_{\gamma_i}X))
	\]
	completing the proof.
\end{proof}

The rest of this section is devoted to the proof of Lemma \ref{lem:alphaIsNatural}.  Recall we have a square
\begin{equation}
\begin{tikzcd}[row sep =large,  column sep = large]
	\burnside_{\mathcal{O}_a}^K  \ar["C^H_K"]{d} \ar["i_K"]{r}& \Tburnside^K_{\compPair} \ar["I^H_K"]{d}\\
	\burnside_{\mathcal{O}_a}^H \ar["i_H"]{r}& \Tburnside^H_{\compPair}  \ar[shorten =4mm, Rightarrow, from = 1-2, to = 2-1,"\alpha"']
\end{tikzcd}
\end{equation}
of functors that we are trying to show is exact.  For any $K$-set $T$, the component $\alpha_T\colon I^H_Ki_K(T)\to i_HC^H_K(T)$ is represented by the bispan
\[
	H\times_{K}T\xleftarrow{H\times_{K}(\epsilon^C_{T})} H\times_K\mathrm{Res}^H_K\map_{K}(H,T) \xrightarrow{\epsilon^I_{\map_K(H,T)}}  \map_{K}(H,T) \xrightarrow{=} \map_{K}(H,T)
\]
where $\epsilon^C$ and $\epsilon^I$ are the counits of the coinduction-restriction and induction restriction adjunctions respectively. In order to keep the presentation a bit more organized, we abuse notation slightly and replace, whenever it won't cause confusion, all instances of induction, restriction, and coinduction functors by the symbols $I$, $\rho$, and $C$ respectively so that the bispan $\alpha$ becomes simply
\begin{equation}\label{eq:alpha}
	IT\xleftarrow{I\epsilon_{T}^C} I\rho C T \xrightarrow{\epsilon_{CT}^I} CT.
\end{equation}

\begin{lem}[Lemma \ref{lem:alphaIsNatural}] \label{lem:secondTechnicalLemma}
	The maps $\alpha_T\colon I^H_Ki_K(T)\to i_HC^H_K(T)$ defined by \eqref{eq:alpha} assemble into a natural transformation $\alpha\colon i_HC^H_K\Rightarrow I^H_Ki^K$.
\end{lem}

\begin{proof}
	We need to show that the maps $\alpha_T$ are sufficiently natural.  Suppose we are given a morphism $\omega\colon T\to T'$ in $\burnside^K_{\mathcal{O}_a}$ which is represented by the span $T\xleftarrow{f} A \xrightarrow{g} T'$.  We need to show that the following square commutes in $\Tburnside^H_{\compPair}$:
	\begin{equation}\label{eq:squareToCommute}
		\begin{tikzcd}[row sep =large, column sep = large]
			IT \ar["Ii_K(\omega)"]{r} \ar["\alpha_T"]{d}&  IT' \ar["\alpha_{T'}"]{d}\\
			CT \ar["i_HC(\omega)"]{r} & CT'
		\end{tikzcd}
	\end{equation}
	
	To evaluate the top right composite, fix a choice of pullback diagram
	\begin{equation}\label{eq:aPullback}
		\begin{tikzcd}
			P \ar["u"]{r} \ar["v"]{d} &\rho CT' \ar["\epsilon^C"]{d} \\
			A \ar["g"]{r} & T'
		\end{tikzcd}
	\end{equation}
	and consider the diagram
	\begin{equation}\label{eq:composite1}
		\begin{tikzcd}[row sep = large,column sep = large]
			IT & IA \ar["I(f)"']{l} \ar["="]{r} & IA \ar["I(g)"]{r} & IT' \\
			& & IP \ar["I(v)"]{u} \ar["I(u)"]{r}& I\rho C T'\ar["I(\epsilon^C)"']{u} \ar["\epsilon^I"]{d}\\
			& I\rho CA \ar["F"]{ur} \ar["\epsilon^I"]{r}& CA \ar["Cg"]{r} & CT'\ar[=]{d}\\
			& & & CT' \ar[from = 1-4,to=2-3,phantom,"(1)"] \ar[from = 2-4,to=3-3,phantom,"(2)"]
		\end{tikzcd}
	\end{equation}
	where the square (1) is the result of applying $H\times_K(-)$ to the pullback \eqref{eq:aPullback} and is thus a pullback. We claim the trapezoid (2) is an exponential diagram for the composable arrows $IP\xrightarrow{I(u)} I\rho CT' \xrightarrow{\epsilon^I}CT'$ and that the map $F$ is such that $I(v)\circ F = I(\epsilon^C_{A})$. We defer the proof of both claims to Lemma \ref{lem:exponentialLemma} below.  
	
	The diagram \eqref{eq:composite1}, and the composition laws in the category $\Tburnside^H_{\compPair}$ tells us the composite $\alpha_{T'}\circ I^H_Ki_K(\omega)$ is the bispan
	\begin{equation}\label{eq:composite3}
		IT \xleftarrow{I(f\circ \epsilon^C)} I\rho CA \xrightarrow{\epsilon^I} CA\xrightarrow{C(g)} CT'
	\end{equation}
	
	We compute left-bottom composite of \eqref{eq:squareToCommute} using the diagram
	\begin{equation}\label{eq:composite2}
		\begin{tikzcd}[row sep = large,column sep = large]
			IT &   & &\\
		I\rho CT \ar["I(\epsilon^C)"]{u} \ar["\epsilon^I"']{d}	&  I\rho C A \ar["I\rho C(f)"']{l} \ar["\epsilon^I"]{d}& & \\
		 CT \ar["="']{d}	& CA \ar["C(f)"']{l} \ar["="]{d} & &\\
		CT 	& CA\ar["C(f)"']{l} \ar["="]{r} &  CA\ar["C(g)"]{r} & CT' \ar[from = 2-1,to=3-2,phantom,"(4)"] \ar[from = 3-1,to=4-2,phantom,"(3)"]
		\end{tikzcd}
	\end{equation}
	in which squares (3) and (4) are both pullbacks.  To see that (4) is a pullback, note that for any $X\in \Set^H$, we have an isomorphism $I\rho(X) \cong H/K\times X$ and the morphism $\epsilon^I\colon I\rho(X)\to X$ is the projection map.
	
	The diagram \eqref{eq:composite2}, and the composition laws in $\Tburnside^H_{\compPair}$ tells us the composite $i_HC(\omega)\circ\alpha_{T}$ is the bispan represented by
	\[
	IT \xleftarrow{I(\epsilon^C\circ \rho C(f))} I\rho CA \xrightarrow{\epsilon^I} CA\xrightarrow{C(g)} CT'
	\] 
	which, by the naturality of $\epsilon^C$, is equal to \eqref{eq:composite3}.
\end{proof}
\begin{lem}\label{lem:exponentialLemma}
	The trapezoid (2) from the proof of Lemma \ref{lem:secondTechnicalLemma} is an exponential diagram.  The composite $I(v)\circ F$ is equal to $I(\epsilon_A^C)$.
\end{lem}
\begin{proof}
	Consider the commutative diagram
	\[
	\begin{tikzcd}
	Q \ar["\gamma"]{r} \ar["\delta"]{d}& CP \ar["C(v)"]{r} \ar["C(u)"]{d} & CA\ar["C(g)"]{d}\\
	CT' \ar["\eta^C"]{r} & C\rho C T' \ar["C(\epsilon^C)"]{r} & CT'
	\end{tikzcd}
	\]
	in which $Q$ is chosen so that left square is a pullback.  Since the right square is $C$ applied to the pullback square \eqref{eq:aPullback}, and coinduction preserves pullbacks, the outside rectangle is a pullback.  Moreover, since the composite along the bottom is the identity by a triangle identity, we could have chosen $Q$ so that $Q=CA$, $\delta = C(g)$, and $C(v)\circ \gamma=\id_{CA}$.  
	
	Lemma 2.3.5 of \cite{Hoyer} says precisely that there is an exponential diagram
	\[
	\begin{tikzcd}
	& I\rho Q \ar["\epsilon^I"]{r}\ar["I(\widehat{\gamma})"']{dl}& Q \ar["\delta"]{d} \\
	IP \ar["I(u)"']{r} & I\rho CT' \ar["\epsilon^I"']{r} & CT'
	\end{tikzcd}
	\]
	where $\widehat{\gamma}$ is the adjunct of $\gamma$ along the coinduction restriction adjunction. Taking $F=I(\widehat{\gamma})$ proves the first claim.  
	
	To prove $I(v)\circ F = I(\epsilon_A^C)$, it suffices to show that $v\circ \widehat{\gamma} = \epsilon_A^C$.  By definition, $\widehat{\gamma}$ is equal to the composite $\epsilon_P^C\circ \rho(\gamma)$ and we have
	\[
	v\circ \widehat{\gamma} = v\circ \epsilon_P^C\circ \rho(\gamma) = \epsilon_A^C\circ \rho C(v)\circ \rho(y) = \epsilon_A^C\circ \id_{\rho CA} = \epsilon^C_A
	\]
	where the third equality uses naturality of $\epsilon^C$ and the fourth follows from our choice of $\gamma$.
\end{proof}

	\bibliographystyle{alpha}
	\bibliography{bibliography}
	
\end{document}